\documentclass[]{elsarticle}
\usepackage{ourstyle}
\allowdisplaybreaks
\title{Approximating Gromov-Hausdorff Distance in Euclidean Space}

\author[berkeley]{Sushovan Majhi\corref{cor1}\fnref{fn1}}
\ead{smajhi@berkeley.edu}

\author[olemiss,tulane]{Jeffrey Vitter}
\ead{jsv@olemiss.edu}

\author[tulane]{Carola Wenk\fnref{fn1}}
\ead{cwenk@tulane.edu}

\fntext[fn1]{supported by the National Science Foundation grant CCF-1618469.}
\address[berkeley]{School of Information, University of California, Berkeley, CA, USA 94720-4600}
\address[olemiss]{Department of Computer and Information Science, University of Mississippi}
\address[tulane]{Department of Computer Science, Tulane University}
\cortext[cor1]{Corresponding author}

\begin{document}

\begin{abstract}
  The Gromov-Hausdorff distance $(d_{GH})$ proves to be a useful distance
  measure between shapes. In order to approximate $d_{GH}$ for compact subsets
  $X,Y\subset\R^d$,  we look into its relationship with $d_{H,iso}$, the infimum
  Hausdorff distance under Euclidean isometries. As already known for dimension
  $d\geq 2$, the $d_{H,iso}$ cannot be bounded above by a constant factor times
  $d_{GH}$. For $d=1$, however, we prove that $d_{H,iso}\leq\frac{5}{4}d_{GH}$.
  We also show that the bound is tight. In effect, this gives rise to an
  $O(n\log{n})$-time algorithm to approximate $d_{GH}$ with an approximation
  factor of $\left(1+\frac{1}{4}\right)$.
\end{abstract}


\maketitle

\section{Introduction}\label{sec:intro} This paper grew out of our effort to
compute the Gromov-Hausdorff distance between Euclidean subsets. The
Gromov-Hausdorff distance between two abstract metric spaces was first
introduced by M. Gromov in ICM 1979 (see Berger \cite{berger_encounter_2000}).
The notion, although it emerged in the context of Riemannian metrics, proves to
be a natural distance measure between any two (compact) metric spaces. Only in
the last decade the Gromov-Hausdorff distance has received much attention from
the researchers in the more applied fields\sush{citation}. In shape recognition
and comparison, shapes are regarded as metric spaces that are deformable under a
class of transformations. Depending on the application in question, a suitable
class of transformations is chosen, then the dissimilarity between the shapes
are defined by a suitable notion of \emph{distance measure or error} that is
invariant under the desired class of transformations. For comparing Euclidean
shapes under Euclidean isometry, the use of Gromov-Hausdorff distance is
proposed and discussed in
\cite{memoli_theoretical_2005,memoli_use_nodate,memoli_gromov-hausdorff_2008,
memoli_properties_2012}.

In this paper, we are primarily motivated by the questions pertaining to the
computation of the Gromov-Hausdorff distance, particularly between Euclidean
subsets.  Although the distance measure puts the Euclidean shape matching on a
robust theoretical foundation \cite{memoli_theoretical_2005,memoli_use_nodate},
the question of computing the Gromov-Hausdorff distance, or even an
approximation thereof, still remains elusive. In the recent years, some efforts
have been made to address such computational aspects. Most notably, the authors
of \cite{agarwal_computing_2015} show an NP-hardness result for approximating
the Gromov-Hausdorff distance between metric trees. For Euclidean subsets,
however, the question of a polynomial-time algorithm is still open.
In~\cite{memoli_properties_2012}, the authors show that computing the
Gromov-Hausdorff distance is related to various NP-hard problems and study a
variant of the distance measure.

The authors of \cite{hutchison_approximating_2005} introduce the additive
distortion---the one used in Gromov-Hausdorff distance---and consider the
problem of minimizing the distortion of bijective functions between sets of the
same cardinality. For the real-line case ($d=1$), the authors demonstrate a
polynomial-time $2$-approximation algorithm \cite[Theorem
6]{hutchison_approximating_2005}. An open problem is also posed in
\cite{hutchison_approximating_2005}: are there polynomial-time approximations to
find the distortion-minimizing bijection withing a factor less than $2$?
Although the context is similar, we consider minimizing the distortion of
correspondences between sets of possibly different cardinalities. As we see in
\defref{cor}, a correspondence between two sets is a more general relation than
a function.

All these computational concerns provoke the natural curiosity about the genuine
computational hardness of the Gromov-Hausdorff distance between Euclidean
subsets. Let $d\geq1$ and $X,Y\subseteq\R^d$ be compact sets equipped with the
standard Euclidean metric, and let $d_{GH}(X,Y)$ denote their Gromov-Hausdorff
distance. One then wonders:
\begin{enumerate}[(i)]
  \item Is there an algorithm to compute $d_{GH}(X,Y)$ \emph{exactly}
  in~polynomial-time?
  \item If not, can we find a polynomial-time approximation algorithm
  for $d_{GH}(X,Y)$, possibly with a reasonably small approximation factor?
  \item If not, is it NP-hard to approximate $d_{GH}(X,Y)$, like the metric
  graph case?
\end{enumerate} 

The above questions motivate our investigation into the computation of $d_{GH}$
in Euclidean spaces. 

\paragraph{Background and Related Work}
The notion of Gromov-Hausdorff distance is closely related to the notion of
Hausdorff distance. Let $(Z,d_Z)$ be any metric space. We first give a formal
definition of the directed Hausdorff distance between any two subsets of $Z$.
\begin{definition}[Directed Hausdorff Distance]\label{def:dh} For any two
 compact subsets $X,Y$ of a metric space $(Z,d_Z)$, the \bemph{directed
 Hausdorff distance} from $X$ to~$Y$, denoted $\overrightarrow{d}_H^Z(X,Y)$, is
 defined by $$\max_{x\in X}\min_{y\in Y}d_Z(x,y).$$
\end{definition}
Unfortunately, the directed Hausdorff distance is not symmetric. To retain
symmetry, the \bemph{Hausdorff distance} is defined in the following way:
\begin{definition}[Hausdorff Distance]\label{def:h} For any two compact subsets
$X,Y$ of a metric space $(Z,d_Z)$, their \bemph{Hausdorff distance}, denoted by
$d_H^Z(X,Y)$, is defined by
$$\max\left\{\overrightarrow{d}_H^Z(X,Y),\overrightarrow{d}_H^Z(Y,X)\right\}.$$
\end{definition}
\noindent To keep our notations simple, we drop the superscript when it is
understood that $Z$ is taken to be $\R^d$ and $X,Y$ are Euclidean subsets
equipped with the standard Euclidean metric $\mod{\cdot}$. The $d_{H}$ can be
computed in $O(n\log{n})$-time for finite point sets with at most $n$ points or
$n$ line segments; see \cite{Alt1995}.

We follow Gromov's book (\cite{gromov_metric_2007}) to define the
Gromov-Hausdorff distance. The primary definition uses the concept of an
isometry or distance-preserving map between metric spaces, which we define
first. 
\begin{definition}[Isometry]
A map $f:(X,d_X)\to (Y,d_Y)$ is called an \emph{isometry} if 
$$d_X(x_1,x_2)=d_Y(f(x_1),f(x_2)).$$ 
\end{definition}
\noindent We immediately note that an isometry $f$ is injective, and that
$f:X\to f(X)$ is a~homeomorphism.

We are now in a place to define the Gromov-Hausdorff distance formally. Unlike
the Hausdorff distance, the Gromov-Hausdorff distance is defined between two
abstract metric spaces $(X,d_X)$ and $(Y,d_Y)$ that may not share a common
ambient space. We start with the following formal definition:

\begin{definition}[Gromov-Hausdorff Distance \cite{gromov_metric_2007}]
  \label{def:gh}
 The \bemph{Gromov-Hausdorff distance}, denoted by $d_{GH}(X,Y)$, between two
 metric spaces $(X,d_X)$ and $(Y,d_Y)$ is defined to be
 $$d_{GH}(X,Y)=\inf_{\substack{f:X\to Z \\g:Y\to Z\\ Z}}d_H^Z(f(X),g(Y)),$$
 where the infimum is taken over all isometries $f:X\to Z$, $g:Y\to Z$
 and~metric spaces $(Z,d_Z)$.
\end{definition}

The definition of Gromov-Hausdorff distance may not seem very natural at first
glance---it deserves a bit of explanation. As mentioned earlier, the definition
works for abstract metric spaces $X$ and $Y$, without requiring them to be
embedded in a common ambient metric space. In order to anatomize \defref{gh}, we
first observe that the maps $f,g$ are embedding $X$ and $Y$, respectively, into
a common metric space $(Z,d_Z)$. Since $f,g$ are isometries, the subsets
$f(X),g(Y)$ of $Z$ are isometric to $X$ and $Y$, respectively. As $f(X)$
and~$g(Y)$ are subsets of $Z$, their Hausdorff distance $d^Z_H(f(X),g(Y))$ can
now be considered. The Gromov-Hausdorff distance is defined to minimize (if the
minimum exists) this~$d^Z_H(f(X),g(Y))$, subject to all isometries $f,g$ and
ambient metric space~$(Z,d_Z)$. As a consequence, Gromov-Hausdorff distance is a
distance measure between abstract metric spaces $X$ and $Y$ that is also
invariant under any isometric transformations of $X$ or $Y$. A detour to
\cite{memoli_properties_2012,gromov_metric_2007,burago_course_2001} is suggested
for a detailed treatment of the definitions and properties of Gromov-Hausdorff
distance.

In order to present an equivalent definition of the Gromov-Hausdorff
distance that is computationally viable, we first define the notion of
a~correspondence.
\begin{definition}[Correspondence]\label{def:cor} A \bemph{correspondence} $\C$
  between any two (non-empty) sets $X$ and $Y$ is defined to be a subset
  $\C\subseteq X\times Y$ with the following two properties:
  \begin{enumerate}[i)]
  \item for any $x\in X$, there exists a $y\in Y$ such that $(x,y)\in\C$, and
  \item for any $y\in Y$, there exists an $x\in X$ such that $(x,y)\in\C$.
  \end{enumerate}
\end{definition}
A correspondence $\C$ is a special \emph{relation} that assigns all points of
both $X$ and $Y$ a corresponding point. If the sets $X$ and $Y$ in the
\defref{cor} are equipped with metrics $d_X$ and $d_Y$, respectively, we can
also define the distortion of the~correspondence $\C$.
\begin{definition}[(Additive) Distortion of
  Correspondence]\label{def:distortion} Let $\C$ be a
  correspondence between two metric spaces $(X,d_X)$ and $(Y,d_Y)$, then its
  \bemph{distortion}, denoted~$Dist(\C)$, is defined to be
  $$\sup_{(x_1,y_1),(x_2,y_2)\in\C}\mod{d_X(x_1,x_2)-d_Y(y_1,y_2)}$$
\end{definition}
The distortion $Dist(\C)$ is sometimes called the \emph{additive} distortion as
opposed to the \emph{multiplicative} distortion; see \cite{kenyon_low_2010} for
a definition. In the context of the Gromov-Hausdorff distance, the distortion of
a correspondence measures how much the two metrics are distorted by the $\C$. In
the extreme case, when $\C$ is an isometry, its distortion becomes zero. For
non-empty sets $X,Y$, we denote by $\C(X,Y)$ the set of all correspondences
between $X$ and $Y$. An alternative definition of the Gromov-Hausdorff distance
is given in \lemref{gh-equiv}; see \cite{burago_course_2001} for a proof.
\begin{lemma}\label{lem:gh-equiv} For any two compact metric spaces $(X,d_X)$
  and $(Y,d_Y)$, the following holds:
  $$d_{GH}(X,Y)=\frac{1}{2}\inf\limits_{\C\in\C(X,Y)} Dist(\C).$$
\end{lemma} 

This combinatorial formulation unveils the genuine complexity entailed in the
computation of Gromov-Hausdorff distance. For two finite metric spaces $X,Y$
containing at most $n$ points, the computation takes $O(2^n)$-time by
enumerating all possible correspondences between the points of $X$ and $Y$.

\paragraph{Our Contribution}
 Our main contribution in this work is to provide a satisfactory answer to the
 quest of understanding the relation between $d_{H,iso}$ (see \defref{hiso}) and
 $d_{GH}$ when $X,Y$ are compact subsets of $\R^1$ equipped with the standard
 Euclidean metric.
 
In \thmref{gh}, we show that $$d_{H,iso}(X,Y)\leq\frac{5}{4}d_{GH}(X,Y)$$ for
any compact $X,Y\subset\R^1$. For subsets of the real line, it is tempting to
believe that $d_{GH}=d_{H,iso}$. We show in \thmref{lb} that this is, in fact,
not true by showing that the bound $\frac{5}{4}$ in \thmref{gh} is tight.  Since
$d_{H,iso}(X,Y)$ can be computed in $O(n\log{n})$-time (\cite{ROTE1991123}), we
provide an $O(n\log{n})$-time algorithm to approximate $d_{GH}(X,Y)$ for finite
$X,Y\subset\R^1$ with an approximation factor of~$(1+\frac{1}{4})$.

\section{Gromov-Hausdorff vs Hausdorff Distance under Isometry}
With the basic definitions now at our disposal, we make our readers acquainted
with a related notion $d_{H,iso}(X,Y)$ here, and list a few of its relevant
consequences. In this section, we also introduce the concept of nearest neighbor
correspondences, present some of their properties, and then illuminate the trail
that has led us to the pinnacle of our findings of \secref{sec:1d}.

\subsection{Comparing $d_{GH}$ with $d_{H,iso}$}
For any dimension $d\geq1$, a \emph{Euclidean isometry}\index{Euclidean
isometry} $T:\R^d\to\R^d$ is defined to be a map that preserves the distance,
i.e.,
$$\mod{T(x_1)-T(x_2)}=\mod{x_1-x_2}\mbox{ for all } x_1,x_2\in\R^d.$$ 

When $d=1$, the map $T$ can only afford to be a translation or a reflection
(flip). In~$d=2$, a Euclidean isometry is characterized by a combination of a
translation, a rotation by an angle, and a reflection about the origin. For more
about Euclidean isometries, see \cite{artin2011algebra}. We denote by $\E(\R^d)$
the set of all isometries of $\R^d$.

\begin{definition}[Hausdorff under Isometry]\label{def:hiso}
For any two compact subsets $X,Y$ of $\R^d$, we define
$$d_{H,iso}(X,Y)=\inf\limits_{T\in\mathcal{E}(\R^d)} d_H(X,T(Y)).$$
\end{definition}

We immediately note that $d_{H,iso}$ induces a pseudo-metric on the set of
compact subsets of $\R^d$; if $d_{H,iso}(X,Y)=0$, then $X$ is \emph{congruent}
to $Y$. \\

\begin{remark}\label{rem:hiso} If $X,Y$ are subsets of $\R^1$ with at most $n$
  points, the authors of \cite{ROTE1991123} prove that their $d_{H,iso}(X,Y)$
  can be computed in $O(n\log{n})$-time.
\end{remark}

The $d_{H,iso}(X,Y)$ minimizes the Hausdorff distance over only Euclidean
isometries; whereas $d_{GH}(X,Y)$ considers minimizing over \emph{all}
isometries and \emph{all} embeddings for $X$ and $Y$---not just Euclidean. The
observation quickly yields the following inequality:
$$
d_{GH}(X,Y)\leq d_{H,iso}(X,Y).
$$

It is most natural to wonder if they are, in fact, equal. To our
disappointment, we can contrive the following configuration in $\R^2$ to show
the contrary in~\exref{gh-2d}. 
\begin{example}[$d_{GH}<d_{H,iso}$ in $\R^2$]\label{ex:gh-2d} Let us consider
  two finite sets $X,Y\subset~\R^2$ as shown in \figref{gh-2d}, where
  $\alpha,h,\eps>0$ are constants. We take $X=\{x_1,x_2,x_3,x_4,x_5\}$,
  $Y=\{y_1,y_2,y_3,y_4,y_5\}$, and~$x_1=y_1$, $x_2=y_2$. In a moment's
  reflection, we see that the blue edges give us the correspondence with minimum
  distortion:
$$
\C_{opt}=\big\{(x_1,y_1),(x_2,y_2),(x_4,y_3),(x_5,y_3),(x_3,y_4),(x_3,y_5)\big
\}.
$$ 
with $dist(\C_{opt})=h+\eps$. Consequently, $d_{GH}(X,Y)=\frac{h+\eps}{2}$. On
the other hand,~$d_{H,iso}(X,Y)=h$. So, for $\eps<h$, we have
$d_{GH}(X,Y)<d_{H,iso}(X,Y)$.

\begin{figure}[thb]
  \centering
  \includegraphics[scale=1]{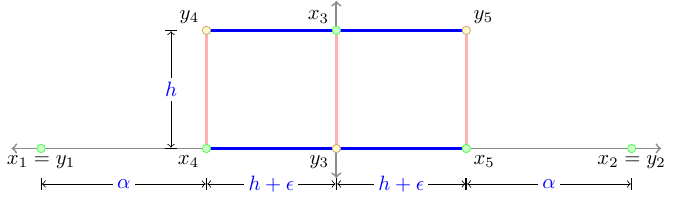}
  \caption[$d_{GH}$ vs $d_{H,iso}$ in Higher Dimensions]
  {The points of $X$ and $Y$ are shown in green and yellow, respectively. The
  correspondence with minimum distortion is shown by the blue edges and the
  nearest neighbor correspondence (see \defref{gh-cnn}) is shown by the red
  edges.}
  \label{fig:gh-2d}
\end{figure}
\end{example}

In \cite{memoli_gromov-hausdorff_2008}, M\'emoli shows the following bounds,
relating $d_{H,iso}$ and $d_{GH}$ between two compact subsets $X,Y$ of $\R^d$.
\begin{equation}\label{eqn:memoli}
  d_{GH}(X,Y)\leq d_{H,iso}(X,Y)\leq c'_d(M)^\frac{1}{2}\sqrt{d_{GH}(X,Y)},
\end{equation}
where $M=\max\big\{\mbox{diameter}(X),\mbox{diameter}(Y)\big\}$ and $c'_d$ is a
constant that depends only on the dimension $d$. In the inequality
\eqnref{memoli}, note the upper bound depends on the diameter of the input sets
$X$ and $Y$. For $d\geq 2$, such a dependence is unavoidable. See
\figref{gh-2d}. This leaves us with $d=1$, the compact subsets of the real line.

In $\R^1$ it often helps to visualize $X\times Y$ on the disjoint union of two
real lines in $\R^2$ and a correspondence $\C\in\C(X,Y)$ by edges between the
corresponding points; see \figref{visual}. Such a two dimensional visualization
comes in handy for the proofs.
\begin{figure}[thb]
    \begin{subfigure}[b]{0.5\linewidth}
    \centering
    \includegraphics[scale=0.75]{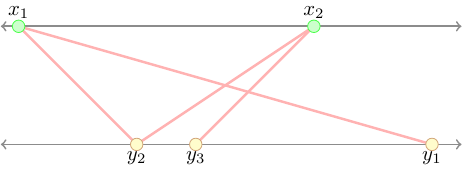}
    \caption{An example correspondence}
    \label{subfig:visual}
    \end{subfigure}
    \begin{subfigure}[b]{0.4\linewidth}
    \centering
    \includegraphics[scale=0.73]{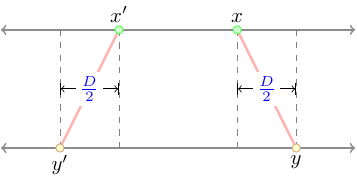}
    \caption{The standard configuration}
    \label{subfig:standard}
    \end{subfigure}
    \caption{On the left, the (sorted) $X=\{x_1,x_2\}$ and $Y=\{y_1,y_2,y_3\}$
    are identified as subsets of the top and the bottom lines respectively. The
    points of $X$ are shown in green, and the points of $Y$ are shown in yellow.
    We visualize the correspondence $\C=\{(x_1,y_1),(x_2,y_2),(x_2,y_3)\}$ by
    the red edges between the respective points. Also, the edges $(x_1,y_1)$ and
    $(x_2,y_2)$ are crossing. On the right, the distortion $D$ of a
    correspondence is attained by the pairs $(x',y')$ and $(x,y)$.
    \label{fig:visual}}
\end{figure}

It is tempting to believe, due to the deceptively simple structure of the real
line, that $d_{GH}=d_{H,iso}$. If it was true, we could compute $d_{GH}$ in
near-linear time; see \remref{hiso}. However, the following sophisticated
construction shows that the conjecture is false.
\begin{example}[$d_{GH}<d_{H,iso}$ in $\R^1$]\label{ex:gh-1d} In this example,
  we show that for any given $\delta>0$, there exist compact $X,Y\subset\R^1$
  such that $d_{GH}(X,Y)=\delta$ and $d_{H,iso}(X,Y)=\delta+\frac{\delta}{8}$.
  As a consequence, $d_{GH}(X,Y)<d_{H,iso}(X,Y)$.
  
  \begin{figure}[thb]
      \centering
      \includegraphics[scale=0.74]{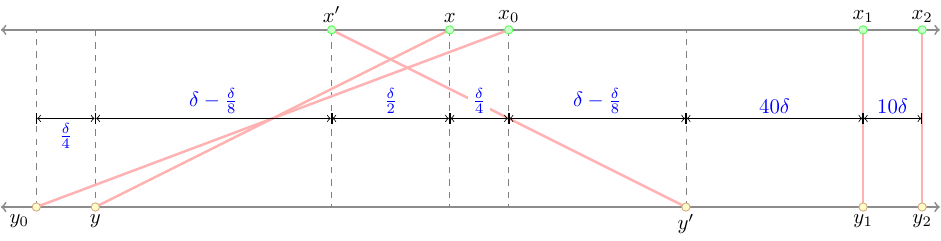}
      \caption[$d_{GH}$ vs $d_{H,iso}$ in $R^1$]
      {The points of $X$ and $Y$ are shown in green and yellow, respectively, on
      two copies of the real line. The optimal correspondence is shown by the
      red edges. The distortion for the correspondence is $2\delta$,
      consequently $d_{GH}(X,Y)=\delta$. We also note that the optimal
      correspondence is not crossing free.}
      \label{fig:gh_vs_h}
  \end{figure}
  The subsets $X,Y$ are taken as shown in \figref{gh_vs_h}, and $\delta>0$. We
  note that $d_H(X,Y)=\delta+\frac{\delta}{8}$. Now, we claim that
  $d_{H,iso}(X,Y)=\delta+\frac{\delta}{8}$. For a proof of our claim, we present
  in \tabref{lb1} the summary of $d_H(X,Y+\Delta)$, which considers translations
  of $Y$ by an amount $\Delta\in\R^1$. We also note that a translation of $-Y$
  does not help to reduce the~Hausdorff distance.
  
  {
  \setlength{\tabcolsep}{12pt}
  \renewcommand{\arraystretch}{1.7}
  \begin{table}[tbh]
      \centering
      \begin{tabular}{|c|p{2.5cm}|p{2.5cm}|c|}
          \hline
           $\Delta$ & $\overrightarrow{d}_H(X,Y+\Delta)$ 
           & $\overrightarrow{d}_H(Y+\Delta,X)$ & $d_H(X,Y+\Delta)$\\ 
           \hline
           $(-\infty,0)$ & -- & $(y_0,x')$ & $>\delta+\frac{\delta}{8}$ \\ 
           \hline
           $0$ & $(x,y')$ & $(y_0,x')$ & $\delta+\frac{\delta}{8}$\\
           \hline
           $(0,\frac{\delta}{8})$ & -- & $(y_k,x),(y_k,x_k)$ 
           & $\in\left(\delta+\frac{\delta}{8},\delta+\frac{\delta}{4}\right)$\\
           \hline
           $\frac{\delta}{8}$ & $(x,y),(x,y')$ & -- 
           & $\delta+\frac{\delta}{4}$\\
           \hline
           $(\frac{\delta}{8},\frac{\delta}{4})$ & $(x,y)$ & -- 
           & $\delta+\frac{\delta}{8}$\\
           \hline
           $\frac{\delta}{4}$ & $(x,y)$ & $(y',x_0)$ 
           & $\delta+\frac{\delta}{8}$\\
          \hline
          $(\frac{\delta}{4},\infty)$ & -- & $(y',x_0)$ 
          & $>\delta+\frac{\delta}{8}$\\
          \hline
      \end{tabular}
      \caption[Summary of $d_H(X,Y+\Delta)$]
      {A summary of $d_H(X,Y+\Delta)$ is recorded for $\Delta\in\R$ for
      \exref{gh-1d} shown in \figref{gh_vs_h}. In the second and third columns,
      the directed Hausdorff distances are achieved for the shown pairs of
      points. The other columns are self-explanatory.}
      \label{tab:lb1}
  \end{table}
  }
\end{example}

By our observation in the above example, we are intrigued by the quest of
bounding $d_{H,iso}$ from above by a constant multiple of $d_{GH}$. A constant
upper bound is presented in \secref{sec:1d} for $d=1$, along with the proof that
the bound is tight.

\subsection{Nearest Neighbor Correspondence}
To conclude our discussion of this section, we lastly present our line of
investigation into Hausdorff correspondences. As noted previously, in
\exref{gh-2d} and \exref{gh-1d}, $d_{GH}(X,Y)\neq d_{H,iso}(X,Y)$ in general.
However, we take the analysis one step further, and explore in \thmref{gh-cnn}
that there does not exist a Euclidean isometry $T$ such that the Hausdorff
correspondence between $X$ and~$T(Y)$ is an optimal correspondence that
minimizes the distortion.

Among many possible correspondences, the following correspondence is
particularly interesting when considering two Euclidean subsets. 
\begin{definition}[Nearest Neighbor Correspondence]\label{def:gh-cnn}
  \index{nearest neighbor correspondence}
  For any two compact subsets~$X,Y\subset\R^N$, we define the \bemph{nearest
  neighbor correspondence} $\C_{NN}$ to be the relation defined by the nearest
  neighbors of points of $X$ and $Y$. More precisely,
  \begin{align*}
    \C_{NN}=\{(x,y)\in X\times Y\mid(x\text{ is a nearest neighbor of }y
    \text{ in }X) \\ 
    \text{ or }(y\text{ is a nearest neighbor of }x\text{ in }Y)\}.
  \end{align*}
\end{definition}
Since $X$ and $Y$ are compact, the nearest neighbors exist. Hence, the
induced relation is a correspondence. Now follows an important structural
property concerning crossings of a correspondence.
\begin{definition}[Crossing]\label{def:cross} For a correspondence
$\C\in\C(X,Y)$, we say a pair of edges $(x_1,y_1)$ and $(x_2,y_2)$ is a
\bemph{crossing} if they cross in the usual sense, i.e., either of the following
happens $x_1<x_2$ but $y_1>y_2$ or $x_1>x_2$ but $y_1<y_2$; see \figref{visual}.
\end{definition}

\begin{lemma}[Crossing]\label{lem:gh-cnn}
    For any two compact $X,Y\subset\R^1$, the nearest neighbor correspondence
    $\C_{NN}(X,Y)$ is free of crossings.
\end{lemma}
  
\begin{proof}
  Let us consider two edges $e=(x,y)$ and $e'=(x',y')$ in $\C_{NN}(X,Y)$
  with~$x<x'$. Without loss of generality, we assume that $y$ is a nearest
  neighbor of $x$. In order to show that $e$ cannot cross $e'$, we assume the
  contrary: $y'<y$. Now, we consider the following cases based on the position
  of $y$ with respect to~$x$. In each case we show that neither $y'$ is a
  nearest neighbor of $x'$ nor $x'$ is a nearest neighbor of $y'$. This would
  contradict the fact that $(x',y')\in\C(X,Y)$.

  \begin{case}[$y\leq x$]
    In this case, a nearest neighbor of $x'$ cannot be smaller than $y$. Hence,
    $y'$ cannot be its nearest neighbor. Also for $y'<y$, its nearest neighbor
    has to be also smaller than $x$, hence $x'$ cannot be its nearest neighbor
    either.
  \end{case}
  
  \begin{case}[$x<y$]
    A nearest neighbor of $x'$ cannot be smaller than $y$. Hence,~$y'$ cannot be
    its nearest neighbor. Since $y'<y$, we have $y'\leq x$ in this case.
    Therefore, a nearest neighbor $y'$ has to be smaller than $x$. So, $x'$
    cannot be its nearest neighbor.
  \end{case}
\end{proof}

We wrap up this section with our final result of this section in the
following theorem.
\begin{theorem}\label{thm:gh-cnn}
  For $d\geq1$, there exist compact subsets $X,Y\subset\R^d$ such that 
  $\C_{NN}(X,T(Y))$ is not an optimal correspondence for any
  Euclidean isometry $T\in\mathcal{E}(\R^d)$.
\end{theorem}
\begin{proof}
  For $d\geq2$, we refer the readers to \exref{gh-2d} and \figref{gh-2d}. We
  also note here that there does not exist any non-trivial Euclidean isometry
  $T$ such that $\C_{opt}$ becomes the nearest neighbor correspondence.
  
  In the case $d=1$, we use $X,Y$ from \exref{gh-1d} and \figref{gh_vs_h}.
  Here,~$\C_{opt}$ must have crossings---even when one considers the reflection
  of $Y$. By \lemref{gh-cnn}, $\C_{opt}$ cannot be produced by any nearest
  neighbor correspondence.
\end{proof}

\section{Approximating Gromov-Hausdorff Distance in $\R^1$}
\label{sec:1d} 
This section is devoted to our main result (\thmref{gh}) on approximating the
Gromov-Hausdorff distance between subsets of the real line. Unless stated
otherwise, in this section we always assume that $X,Y$ are compact subsets
of~$\R^1$ and both are equipped with the standard Euclidean metric denoted by
$\mod{\cdot}$.

\begin{definition}[Standard Configuration]\label{def:standard}
  Let $\C$ be any correspondence between two compact subsets $X,Y$ of $\R^1$.
  There exist $(x,y),(x',y')\in\C$ such that~$\bigmod{\mod{x-x'}-\mod{y-y'}}=D$,
  where $D$ is the distortion of $\C$. Without loss of generality, we assume
  that $x\leq x'$ and $\mod{x-x'}\leq\mod{y-y'}$. Then, there exists an
  $\R^1$-isometry such that, when applied on $Y$, $(x,y)$ and $(x',y')$ are
  arranged as in \figref{visual}(b). For the proofs that follow, we assume this
  standard configuration for any given correspondence $\C$, and denote by
  $\{(x,y),(x',y')\}$ a fixed pair that has this property.
\end{definition}
  
We have already noted that  $d_{GH}(X,Y)\leq d_{H,iso}(X,Y)$ for any compact
$X,Y\subset\R^d$; see \cite{memoli_gromov-hausdorff_2008}. Together with this,
\thmref{gh} thus gives us the approximation algorithm for $d_{GH}$ with an
approximation factor of $\left(1+\frac{1}{4}\right)$. Later in \thmref{lb}, we
also show that the upper bound of \thmref{gh} is tight.
\begin{theorem}[Approximation of Gromov-Hausdorff Distance]\label{thm:gh}
  For any two compact $X,Y\subset\R^1$ we have
  $$d_{H,iso}(X,Y)\leq\frac{5}{4}d_{GH}(X,Y).$$
\end{theorem}
\begin{proof}    
  In order to prove the result, it suffices to show that for any correspondence
  $\C\in\C(X,Y)$ with $Dist(\C)=D$, there exists a Euclidean isometry
  $T\in\mathcal{E}(\R^1)$ such that
  $$d_H(X,T(Y))\leq\frac{5D}{8}.$$ Depending on the crossing (\defref{cross})
   behavior, we classify a given correspondence into three main types: no double
   crossing (\defref{double-cross}), wide crossing (\defref{wide-cross}), and no
   wide crossing, and we divide the proof for each type into \thmref{no-cross},
   \thmref{wide-cross}, and \thmref{no-wide}, respectively.
\end{proof}

\subsection{No Double Crossings}
We start with the definition of a double crossing edge.
\begin{definition}\label{def:double-cross} An edge in a correspondence
$\C\in\C(X,Y)$ is said to be a \bemph{double crossing} if it crosses both the
edges $(x',y')$ and $(x,y)$ as defined in \defref{standard}; see \figref{cross}.
\end{definition}

In the following theorem, we consider the case when there is no double
crossing edge in $\C$.
\begin{theorem}[No Double Crossing]\label{thm:no-cross}
  For a correspondence $\C\in\C(X,Y)$ without any double crossing, there exists
  a value $\Delta\in\R$ such that
  $$d_H(X,Y+\Delta)\leq\frac{5D}{8},$$ where $D=Dist(\C)$.
\end{theorem}

\begin{proof}
  In the trivial case, when $d_H(X,Y)\leq\frac{5D}{8}$, we take $\Delta=0$. We
  now assume the non-trivial case that $d_H(X,Y)>\frac{5D}{8}$. As shown in
  \figref{ub-all}, we define the following two subsets of $X$:
  \begin{equation}\label{eqn:A}
  \begin{split}
    A&=\left\{a\in X\cap\left(x+D,\infty\right)\mid
    \exists~b\in Y\cap[y',y]\mbox{ with }(a,b)\in\C\right\}, \\ 
    A'&=\left\{a\in X\cap\left(-\infty,x'-D\right)
    \mid\exists~b\in Y\cap[y',y]\mbox{ with }(a,b)\in\C\right\}.
  \end{split}
  \end{equation}
  \begin{figure}[thb]
    \centering \includegraphics[scale=0.8]{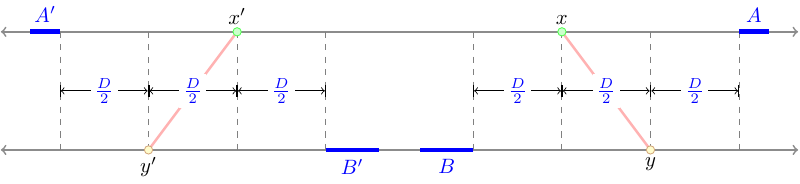} \caption[No Double
    Crossing]{The no double crossing case is shown. The sets $A,A'$ and $B,B'$
    are shown as subsets of the thick, blue intervals in the top and bottom.}
    \label{fig:ub-all}
  \end{figure}
  We also define the following two subsets of $Y$:
  \begin{equation}\label{eqn:B}
  \begin{split}
    B&=\left\{b\in Y\cap\left(y'+D,y-D\right)\mid\exists~a\in X\cap(x,\infty)
    \mbox{ with }
    (a,b)\in\C\right\},\\
    B'&=\left\{b\in Y\cap\left(y'+D,y-D\right)\mid\exists~a\in X\cap(-\infty,x')
    \mbox{ with }
    (a,b)\in\C\right\}.
  \end{split}
  \end{equation}
  As shown in \lemref{AB}, not all pairs of the above defined sets can be
  non-empty together. The assumption that $d_H(X,Y)>\frac{5D}{8}$ implies, again
  from \lemref{AB}, that one of the above sets must be non-empty. Without any
  loss of generality, it then suffices to study only the following three unique
  cases: 
  \begin{enumerate}
  \item $A\neq\emptyset, B=\emptyset$
  \item $A=\emptyset, B\neq\emptyset$
  \item $A\neq\emptyset, B\neq\emptyset$.
  \end{enumerate}
  For each of the cases, we show that a positive number $\Delta$ can be chosen
  such that $d_H(X,Y+\Delta)\leq\frac{5D}{8}$.
  \begin{case}[$A\neq\emptyset,B=\emptyset$]
  We denote $p_0=\max{A}$ and $\eps=p_0-x-D$. We also let $q_0\in Y\cap[y',y]$
  such that $(p_0,q_0)\in\C$; let $\eps'=(y-q_0)$; see \figref{ub-case-1}. From
  the assumption that $d_H(X,Y)>\frac{5D}{8}$, we first note that
  $\eps>\frac{D}{8}$. We also argue that 
  \begin{equation}\label{eqn:eps}
    \eps\leq\eps'\leq D-\eps,\text{ and }\eps\leq\frac{D}{2}.
  \end{equation} 
  \begin{figure}[thb]
    \centering \includegraphics[scale=0.8]{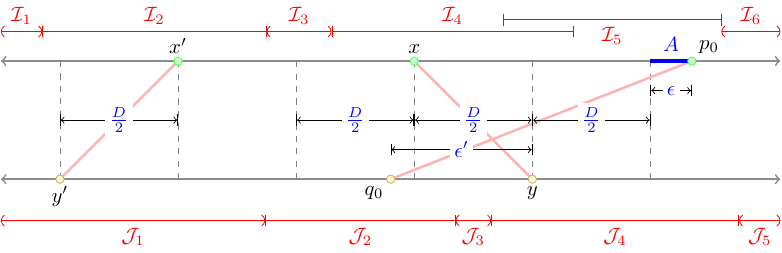} \caption[No Double
    Crossing]{No double crossing Case (1) is shown. The set $A$ is a subset
    of the thick, blue interval in the top.}
    \label{fig:ub-case-1}
  \end{figure}
  To see \eqnref{eps}, we start by observing that $(p_0-x')\geq(q_0-y')$. From
  the observation together with the distortion of the pair $(x',y')$ and
  $(p_0,q_0)$, we get
  $$D\geq\mod{(p_0-x')-(q_0-y')}=(p_0-x')-(q_0-y')=(p_0-q_0)-(x'-y')=\eps+\eps'.$$
  So, $\eps'\leq D-\eps$. In particular, $\eps'\leq D$. So, from the distortion
  of the pair $(x,y)$ and $(p_0,q_0)$, we also get
  $$D\geq\mod{(p_0-x)-(y-q_0)}=\mod{(D+\eps)-\eps'}=D+\eps-\eps'.$$ This implies
  that $\eps\leq\eps'$. As a result, we also have 
  $$\eps\leq\frac{1}{2}(\eps+\eps)\leq\frac{1}{2}(\eps'+\eps)\leq\frac{D}{2}.$$ 

  We now show that $\overrightarrow{d}_H(X,Y+\Delta)\leq\frac{5D}{8}$, for any
  translation amount $\Delta$ such that
  \begin{equation}\label{eqn:delta-1}
    \begin{cases}
    \Delta=\eps-\frac{D}{8},\text{ when }\frac{D}{8}<\eps\leq\frac{D}{4} \\
    \Delta\in\left[\eps-\frac{D}{8},\eps+\eps'-\frac{3D}{8}\right],
    \text{ when }\frac{D}{4}<\eps\leq\frac{D}{2}.
    \end{cases}
  \end{equation}
  Due to \eqnref{eps} and the fact that $\eps>\frac{D}{8}$, we immediately note
  from \eqnref{delta-1} that $$0<\Delta\leq\frac{5D}{8}.$$ In order to cover
  $X$, we consider the following (possibly overlapping) intervals to cover the
  real line:
  \begin{align*}
  \I_1 &= \left(-\infty,y'-\frac{5D}{8}+\Delta\right),
  \I_2 = \left[y'-\frac{5D}{8}+\Delta,y'+\frac{5D}{8}+\Delta\right], \\
  \I_3 &= \left(y'+\frac{5D}{8}+\Delta,q_0-\frac{5D}{8}+\Delta\right), 
  \I_4 = \left[q_0-\frac{5D}{8}+\Delta,q_0+\frac{5D}{8}+\Delta\right], \\
  \I_5 &= \left[y-\frac{5D}{8}+\Delta,y+\frac{5D}{8}+\Delta\right], 
  \text{ and }
  \I_6 = \left(y+\frac{5D}{8}+\Delta,\infty\right).
  \end{align*}
  Since $y-q_0=\eps'\leq D$ from \eqnref{eps}, the intervals $\I_4$ and $\I_5$
  intersect. So, the union of the intervals is the entire real line. For an
  arbitrary point $a\in X$ from any of the above intervals, we show that there
  exists a point $b\in Y$ such that $\mod{a-(b+\Delta)}\leq\frac{5D}{8}$. The
  intervals $\I_2$, $\I_4$, and $\I_5$ are \emph{covered} by $y'$, $q_0$, and
  $y$, respectively. So, we present our argument only for $\I_1,\I_3$, and
  $\I_6$.
  
  \paragraph{($\pmb{\I_1}$)} Let $a\in\I_1\cap X$ and $b\in Y$ such that
  $(a,b)\in\C$. From $\Delta\leq\frac{5D}{8}$, it follows that $a<y'$. We first
  note that $(a,b)$ cannot cross $(x,y)$, since there is no double crossing. We
  next argue that $(a,b)$ does not cross $(p_0,q_0)$, either. We assume the
  contrary that $(a,b)$ crosses $(p_0,q_0)$, consequently $b\in[q_0,y]$. So, we
  get $(b-q_0)\leq\eps'$. And, the distortion of the pair $(a,b)$ and
  $(p_0,q_0)$ yields the following contradiction:
  \begin{align*}
  D&\geq\bigmod{\mod{p_0-a}-\mod{q_0-b}}
  =\mod{(x+D+\eps-a)-(b-q_0)}\\
  &=\mod{(x-a)+D+\eps-(b-q_0)}\\
  &=(x-a)+D+\eps-(b-q_0),\text{ since }(b-q_0)\leq\eps'\leq D
  \text{ from \eqnref{eps}}\\
  &=[(x-x')+(x'-y')+(y'-a)]+D+\eps-\eps' \\
  &\geq(x'-y')+(y'-a)+D+\eps-\eps',\text{ since }(x-x')\geq0 \\
  &=\frac{D}{2}+(y'-a)+D+\eps-\eps' \\
  &\geq\frac{D}{2}+\left(\frac{5D}{8}-\Delta\right)+D+\eps-\eps',\text{ since }a\in\I_1\\
  &=D+\left(\eps-\eps'+\frac{9D}{8}-\Delta\right)\\
  &>D,\text{ since the second term is positive by \lemref{eps}}.
  \end{align*}
  Now that we have $(a,b)$ not crossing $(p_0,q_0)$, we show that
  $\mod{a-(b+\Delta)}\leq\frac{5D}{8}$. If $b\geq a$, then from the distortion
  bound of $D$ for the pair of (non-crossing) edges $(a,b)$ and $(p_0,q_0)$, we
  get $b-a\leq\frac{D}{2}-(\eps+\eps')$. From \eqnref{delta-1}, it is evident
  that $\Delta\leq\eps+\eps'$. So,
  $$\mod{a-(b+\Delta)}=\Delta+(b-a)\leq\Delta+\left[\frac{D}{2}-(\eps+\eps')\right]
  \leq\frac{D}{2}.$$ If $b\leq a$, then from the distortion bound of $D$ for the
  pair of (non-crossing) edges $(a,b)$ and $(x,y)$, we have
  $a-b\leq\frac{D}{2}$. Since we have already observed
  that~$\Delta\leq\frac{5D}{8}$, we also have
  $\mod{a-(b+\Delta)}\leq\frac{5D}{8}$. In either case, we conclude that
  $\mod{a-(b+\Delta)}\leq\frac{5D}{8}$.

  \paragraph{($\pmb{\I_3}$)} We note here that the distance between the
  endpoints of $\I_3$ is
  \begin{align*}  
    &\left(q_0-\frac{5D}{8}+\Delta\right)-\left(y'+\frac{5D}{8}+\Delta\right) \\
    &=(q_0-y')-\frac{10D}{8} \\
    &=(y-\eps')-y'-\frac{10D}{8} \\
    &=(y-y')-\eps'-\frac{10D}{8} \\
    &=(x-x'+D)-\eps'-\frac{10D}{8} \\
    &=(x-x')-\left(\eps'+\frac{D}{4}\right)
  \end{align*}
  So, $\I_3\neq\emptyset$ if and only if $x-x'>\eps'+\frac{D}{4}$. We also note
  similarly that $\I_3\subseteq[x',x]$.  If $\I_3$ is empty, there is nothing
  show. Let us assume that  it is non-empty. 
  
  Let $a\in\I_3\cap X$ and $b\in Y$ such that $(a,b)\in\C$. Since $a\in[x',x]$,
  the edge $(a,b)$ cannot cross $(x,y)$ or $(x',y')$ because of
  \lemref{standard}. Moreover, $\mod{b-a}\leq\frac{D}{2}$. We now argue that
  $(a,b)$ does not cross $(p_0,q_0)$ either. We assume the contrary that the
  edge $(a,b)$ crosses $(p_0,q_0)$, i.e., $q_0<b\leq y$. So, we get the
  following contradiction:
  \begin{align*}
  D&\geq\bigmod{\mod{p_0-a}-\mod{q_0-b}}
  =\mod{(x+D+\eps-a)-(b-q_0)}\\
  &=\mod{(x-a)+D+\eps-(b-q_0)}\\
  &=(x-a)+D+\eps-(b-q_0),\text{ since }(b-q_0)\leq\eps'\leq D
  \text{ from \eqnref{eps}}\\
  &=\left(y-\frac{D}{2}\right)-a+D+\eps-(b-q_0)\\
  &=(q_0+\eps')-\frac{D}{2}-a+D+\eps-(b-q_0)\\
  &=2q_0+\eps'+\frac{D}{2}-a+\eps-b\\
  &\geq2q_0+\eps'+(b-a)-a+\eps-b,
  \text{ since }\frac{D}{2}\geq(b-a)\\
  &=2(q_0-a)+\eps+\eps' \\
  &>2\left(\frac{5D}{8}-\Delta\right)+\eps+\eps',\text{ since }a\in\I_3 \\
  &=D+\left[\frac{D}{4}-2\Delta+\eps+\eps'\right] \\
  &\geq D,\text{ since the second term is non-negative by \lemref{eps-1}}.
  \end{align*}
  Hence, $(a,b)$ does not cross $(p_0,q_0)$. Therefore, by arguing similar to
  $\I_1$ we conclude $\mod{a-(b+\Delta)}\leq\frac{5D}{8}$.
   
  \paragraph{($\pmb{\I_6}$)} For $a\in \I_6\cap X$, we get
  \begin{align*}
  a &>y+\frac{5D}{8}+\Delta=\left(p_0-\eps-\frac{D}{2}\right)+\frac{5D}{8}
  +\Delta \\ &
  =p_0+\frac{D}{8}-\eps+\Delta \\
  &\geq p_0,\text{ since }\Delta\geq\eps-\frac{D}{8}\text{ by \eqnref{delta-1}}\\
  &=\max A.
  \end{align*}
  By the definition of the set $A$, an edge $(a,b)\in\C$ cannot cross $(x,y)$.
  For this reason, it cannot cross $(p_0,q_0)$ either. The rest of the argument
  presented for~$\I_1$ then goes through. Therefore,
  $\mod{a-(b+\Delta)}\leq\frac{5D}{8}$ by arguing similar to $\I_1$.
  
  In order to show the other direction, i.e.,
  $\overrightarrow{d_H}(Y+\Delta,X)\leq\frac{5D}{8}$, we argue that there exists
  a translation amount $\Delta$ satisfying \eqnref{delta-1}. To that end, we
  consider the following (possibly overlapping) partition of the real line into
  intervals:
  \begin{align*}
  \J_1 &= \left(-\infty,x-\frac{5D}{8}-\Delta\right),
  \J_2 = \left[x-\frac{5D}{8}-\Delta,x+\frac{5D}{8}-\Delta\right], \\
  \J_3 &= \left(x+\frac{5D}{8}-\Delta,p_0-\frac{5D}{8}-\Delta\right),
  \J_4 = \left[p_0-\frac{5D}{8}-\Delta,p_0+\frac{5D}{8}-\Delta\right],
  \mbox{ and } \\
  \J_5 &= \left(p_0+\frac{5D}{8}-\Delta,\infty\right).
  \end{align*}
  For an arbitrary point $b\in Y$ from any of the above intervals, we now show
  that there exists a point $a$ in $X$ such that
  $\mod{a-(b+\Delta)}\leq\frac{5D}{8}$. The intervals $\J_2$ and $\J_4$ are
  \emph{covered} by $x$ and $p_0$, respectively. We present our argument only
  for $\J_1,\J_3$, and $\J_5$. While the argument for $\J_1$ and $\J_5$ goes
  through for any $\Delta$ satisfying \eqnref{delta-1}, the interval $\J_3$ may
  require a more particular choice of $\Delta$.
  
  \paragraph{($\pmb{\J_1}$)} For any $b\in \J_1\cap Y$ with an edge
  $(a,b)\in\C$, we argue that the edge cannot cross $(p_0,q_0)$. Since
  $\Delta>0$, we note that either $b\in[y',y-D)$ or $b<y'$. If $b\in[y',y-D]$,
  the edge $(a,b)$ cannot cross $(x,y)$, since $B=\emptyset$. If $b<y'$, the
  edge $(a,b)$ cannot cross $(x,y)$, as no double crossing is allowed. Now,
  $b<y-D\leq q_0$ implies that the edge $(a,b)$ cannot cross $(p_0,q_0)$ without
  first crossing $(x,y)$. The rest of the argument presented for $\I_1$ then
  goes through. Therefore, $\mod{a-(b+\Delta)}\leq\frac{5D}{8}$ by arguing
  similar to $\I_1$.
  
  \paragraph{($\pmb{\J_3}$)} We observe that $\J_3\neq\emptyset$ if and only if
  $\eps>\frac{D}{4}$. There is nothing to show if $\J_3$ is empty. So, we assume
  that $\eps>\frac{D}{4}$, and claim that there must exist a $\Delta$ in
  $\left[\eps-\frac{D}{8},\eps+\eps'-\frac{3D}{8}\right]$, hence satisfying
  \eqnref{delta-1}, such that for any $b\in Y\cap\J_3$ there exists $a\in X$
  with $\mod{a-(b+\Delta)}\leq\frac{5D}{8}$. 
  
  We argue by contradiction. Let us assume that no such $\Delta$ exists in given
  the interval, i.e., for all
  $\Delta\in\left[\eps-\frac{D}{8},\eps+\eps'-\frac{3D}{8}\right]$ the following
  subset of $\J_3$ is non-empty:
  $$
    E(\Delta)=\left\{b\in\J_3\mid X\cap\left[b-\frac{5D}{8}+\Delta,b+\frac{5D}{8}+\Delta\right]=\emptyset\right\}.
  $$
  Let us also define the following subsets of $Y$:
  $$
  Y_\leq=\left\{b\in Y\mid\text{there exists an edge }(a,b)\in\C\text{ with }a\leq b\right\}
  $$
  and
  $$
  Y_\geq=\left\{b\in Y\mid\text{there exists an edge }(a,b)\in\C\text{ with }a\geq b\right\}.
  $$
  We note from \lemref{LR} that $E(\Delta)\cap Y_\leq$ and $E(\Delta)\cap
  Y_\geq$ cannot be both non-empty for any given $\Delta$. Therefore,
  $E(\Delta)\neq\emptyset$ for all $\Delta$ implies
  \begin{enumerate}[i)]
    \item either $E(\Delta)\cap Y_\leq$ is empty and $E(\Delta)\cap Y_\geq$
    non-empty for all $\Delta$, or
    \item $E(\Delta)\cap Y_\leq$ is non-empty and $E(\Delta)\cap Y_\geq$ empty
    for all $\Delta$. 
  \end{enumerate}
  In order to arrive at a contradiction, we now show, however, that
    $E(\Delta)\cap Y_\geq$ is empty when $\Delta=\eps-\frac{D}{8}$, whereas
    $E(\Delta)\cap Y_\leq$ is empty when $\Delta=\eps+\eps'-\frac{3D}{8}$.
  
  When $\Delta=\eps-\frac{D}{8}$, we show that $E(\Delta)\cap Y_\geq=\emptyset$.
  Consider any $b\in\J_3\cap Y_\geq$ and an edge $(a,b)\in\C$ such that $a\geq
  b$. Due to the distortion bound of $D$ for the pair of edges $(x,y)$ and
  $(a,b)$, we have $(a-x) \leq D + (y-b)$. So,
  \begin{align*}  
    a-(b+\Delta)
    &\leq (x + D + y -b )- (b+\Delta) \\
    &\leq x + D + y -2b-\Delta \\
    &< x + D + y-2\left(x+\frac{5D}{8}-\Delta\right)-\Delta,\text{ since }b\in\J_3 \\
    &=D + (y - x) -\frac{10D}{8}+\Delta \\
    &=D+\frac{D}{2}-\frac{10D}{8}+\Delta \\
    &=\frac{D}{4}+\left(\eps-\frac{D}{8}\right) \\
    &=\eps+\frac{D}{8}\leq\frac{5D}{8},\text{ since }\eps\leq\frac{D}{2}
    \text{ from \eqnref{eps}}.
  \end{align*}
  So, $b\notin E(\Delta)$. Therefore, $E(\Delta)\cap Y_\geq=\emptyset$ for
  $\Delta=\eps-\frac{D}{8}$. 
  
  We show now that $E(\Delta)\cap Y_\leq=\emptyset$ when
  $\Delta=\eps+\eps'-\frac{3D}{8}$. To see this, consider $b\in \J_3\cap Y_\leq$
  and an edge $(a,b)\in\C$ such that $a\leq b$. Due to the distortion bound of
  $D$ for the pair of edges $(p_0,q_0)$ and $(a,b)$, we have $(p_0-a) \leq D +
  (b-q_0)$. So,
  \begin{align*}
    (b+\Delta)-a 
    &\leq(b+\Delta) + D + (b-q_0) - p_0\\    
    &\leq\Delta + D + 2b-q_0- p_0\\
    &<\Delta+D+2\left(p_0-\frac{5D}{8}-\Delta\right)-q_0-p_0,\text{ since }b\in\J_3  \\
    &=(p_0-q_0)-\frac{2D}{8}-\Delta \\
    &=\left(\eps+\frac{D}{2}+\eps'\right)-\frac{2D}{8}-\left(\eps+\eps'-\frac{3D}{8}\right) \\
    &=\frac{5D}{8}.
  \end{align*}
  So, $b\notin E(\Delta)$. Therefore, $E(\Delta)\cap Y_\leq=\emptyset$ for
  $\Delta=\eps+\eps'-\frac{3D}{8}$. This is a contradiction. Hence, there must
  exist $\Delta$ satisfying \eqnref{delta-1}.
  
  \paragraph{($\pmb{\J_5}$)} Similarly for $b\in\J_5\cap Y$, an edge
  $(a,b)\in\C$ cannot cross $(x,y)$ due to \lemref{standard-1}. Moreover, it
  cannot cross $(p_0,q_0)$. If it did, i.e., $a\in[x,p_0]$, the distortion of
  the pair would exceed $D$:
  \begin{align*}
    &\mod{(b-q_0)-(p_0-a)} \\
    &\geq(b-q_0)-p_0+\left(b-\frac{D}{2}\right),
    \text{ as }(b-a)\leq\frac{D}{2},\text{ from \lemref{standard-1}} \\
    &=(b-q_0)+(b-p_0)-\frac{D}{2}=[(b-p_0)+(p_0-q_0)]+(b-p_0)-\frac{D}{2} \\
    &=(p_0-q_0)+2(b-p_0)-\frac{D}{2}
    =\left(\eps+\eps'+\frac{D}{2}\right)+2(b-p_0)-\frac{D}{2} \\
    &>\left(\eps+\eps'+\frac{D}{2}\right)+2\left(\frac{5D}{8}-\Delta\right)
    -\frac{D}{2},\text{ as }b\in\J_5 \\
    &=\frac{5D}{4}+\eps+\eps'-2\Delta\geq\frac{5D}{4}+\eps+\eps'
    -2\left(\eps+\eps'-\frac{3D}{8}\right),\text{ from \eqnref{delta-1}} \\
    &=2D-(\eps+\eps')\geq D,\text{ as }\eps+\eps'\leq D\text{ from \eqnref{eps}}. 
  \end{align*}
  So, $(a,b)$ does not cross $(p_0,q_0)$. Hence, following the argument
  presented in $\I_1$, we conclude $\mod{a-(b+\Delta)}\leq\frac{5D}{8}$.
  \end{case}
  
  \begin{case}[$A=\emptyset,B\neq\emptyset$]
  We denote $q_1=\min{B}$ and $\eta=y-D-q_1$. We also let $p_1\in
  X\cap(x,\infty)$ such that $(p_1,q_1)\in\C$ and $\eta'=p_1-x$. From the
  assumption that $d_H(X,Y)>\frac{5D}{8}$, we first note that
  $\eta>\frac{D}{8}$. We also argue that 
  \begin{figure}[thb]
    \centering \includegraphics[scale=0.8]{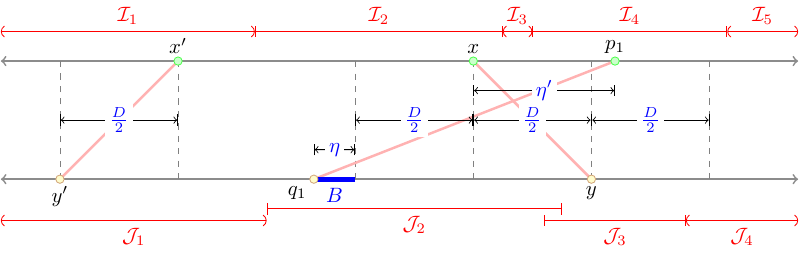} \caption[No Double
    Crossing]{No double crossing Case (2) is shown. The set $B$ is a subset of
    the thick, blue interval in the bottom.}
    \label{fig:ub-case-2}
  \end{figure}
  \begin{equation}\label{eqn:eta}
    \eta\leq\eta'\leq D-\eta,\text{ and }\eta\leq\frac{D}{2}.
  \end{equation} 

  To see \eqnref{eta}, we first observe that $(p_1-x')\geq(q_1-y')$. From the
  observation together with the distortion of the pair $(x',y')$ and
  $(p_1,q_1)$, we get
  \begin{align*}
  D&\geq\mod{(p_1-x')-(q_1-y')}=(p_1-x')-(q_1-y'),\mbox{ since }(p_1-x')
  \geq(q_1-y')\\
  &=\left[(x-x')+\eta'\right]-\left[\frac{D}{2}+(x-x')-\frac{D}{2}
  -\eta\right]=\eta'+\eta.
  \end{align*}
  In particular, $\eta'\leq D$. Now from the distortion of the pair $(x,y)$ and
  $(p_1,q_1)$, we also get
  $$D\geq\mod{(y-q_1)-(x-p_1)}=\mod{D+\eta-\eta'}=D+\eta-\eta'.$$ This implies
  that $\eta'\geq\eta$. Combining this with $\eta+\eta'\leq D$, we get
  $\eta\leq\frac{D}{2}$. Comparing \eqnref{eps} and \eqnref{eta}, we note that
  $\eta$ and $\eta'$  are analogous to $\eps$ and $\eps'$ from Case ($1$),
  respectively. Also, the edge $(p_1,q_1)$ is analogous to the edge $(p_0,q_0)$
  in Case ($1$). Using this analogy, we can reuse the arguments presented in
  Case ($1$) for some of the intervals here.

  We now show that $\overrightarrow{d}_H(Y+\Delta,X)\leq\frac{5D}{8}$ for any
  translation amount $\Delta$ such that
  \begin{equation}\label{eqn:delta-2}
    \begin{cases}
    \Delta=\eta-\frac{D}{8},\text{ when }\frac{D}{8}<\eta\leq\frac{D}{4} \\
    \Delta\in\left[\eta-\frac{D}{8},\eta+\eta'-\frac{3D}{8}\right],
    \text{ when }\frac{D}{4}<\eta\leq\frac{D}{2}.
    \end{cases}  
  \end{equation}
  Due to \eqnref{eta} and the fact that $\eta>\frac{D}{8}$, we immediately note
  that $$0<\Delta\leq\frac{5D}{8}.$$ We consider the following intervals:
  \begin{align*}
  \J_1 &= \left(-\infty,x-\frac{5D}{8}-\Delta\right),
  \J_2 = \left[x-\frac{5D}{8}-\Delta,x+\frac{5D}{8}-\Delta\right], \\
  \J_3 &= \left[p_1-\frac{5D}{8}-\Delta,p_1+\frac{5D}{8}-\Delta\right],
  \text{ and }\J_4 = \left(p_1+\frac{5D}{8}-\Delta,\infty\right).
  \end{align*}
  Since $p_1-x=\eta'\leq D$ from \eqnref{eta}, the intervals $\J_2$ and $\J_3$
  intersect. So, the union of the intervals is the entire real line. For an
  arbitrary $b\in Y$ from any of the above intervals, we show there exists a
  point $a\in X$ such that $\mod{a-(b+\Delta)}\leq\frac{5D}{8}$ for any
  translation $\Delta$ satisfying \eqnref{delta-2}. The intervals $\J_2$ and
  $\J_3$ are \emph{covered} by $x$ and $p_1$, respectively. The nearest neighbor
  argument for the intervals $\J_1$ and~$\J_4$ are analogous to the intervals
  $\I_6,\I_3$ as presented in Case (1), respectively.
  
  Now, in order to show $\overrightarrow{d_H}(X,Y+\Delta)\leq\frac{5D}{8}$, we
  consider the following (possibly overlapping) partition of the real line into
  intervals:
  \begin{align*}
  \I_1 &=\left(-\infty,q_1-\frac{5D}{8}+\Delta\right), 
  \I_2 =\left[q_1-\frac{5D}{8}+\Delta,q_1+\frac{5D}{8}+\Delta\right], \\
  \I_3 &=\left(q_1+\frac{5D}{8}+\Delta,y-\frac{5D}{8}+\Delta\right),
  \I_4 =\left[y-\frac{5D}{8}+\Delta,y+\frac{5D}{8}+\Delta\right], \\
  \I_5 &=\left(y+\frac{5D}{8}+\Delta,\infty\right). 
  \end{align*}
  Using the analogy we mentioned, the intervals $\I_1,\I_2,\I_3,\I_4$, and
  $\I_5$ can be reasoned about analogous to $\J_5,\J_4,\J_3,\J_2$, and $\J_1$
  from Case ($1$), respectively. Also, we note that $B=\emptyset$ in Case ($1$)
  is analogous to having $A=\emptyset$ in this case; the argument for $\I_5$
  here uses the fact that $A=\emptyset$. 
  \end{case}
  
  \begin{case}[$A\neq\emptyset,B\neq\emptyset$]
  The case is depicted in \figref{ub-case-3}. The choice of $\Delta$ in this
  case depends on how $\eps$ compares to $\eta$.
  \begin{figure}[thb] \centering \includegraphics[scale=0.8]{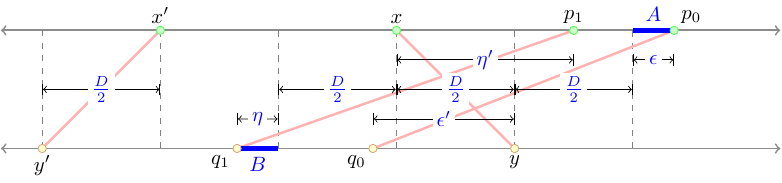}
  \caption[No Double Crossing]{No double crossing Case (3) is shown. The
  sets $A,B$ are subsets of the thick, blue intervals in the top and bottom,
  respectively.}
    \label{fig:ub-case-3}
  \end{figure}
  If $\eps\geq\eta$, we choose $\Delta$ according to Case ($1$). In Case ($1$),
  the argument for the interval $\J_1$ relies on the fact that~$B=\emptyset$.
  Since $\eps\geq\eta$, we have
  $$\Delta\geq\eps-\frac{D}{8}\geq\eta-\frac{D}{8}.$$ Consequently, for any
  $b\in\J_1\cap Y$, we have $b<q_1$. By the definition of the set~$B$ and the
  point $q_1$, any edge $(a,b)$ then cannot cross $(x,y)$, which is analogous to
  having $B=\emptyset$ for Case ($1$).

  Similarly, if $\eta\geq\eps$, we choose $\Delta$ according to Case ($2$). In
  Case ($2$), the argument for the interval $\I_5$ relies on the fact that
  $A=\emptyset$. Since $\eta\geq\eps$, we have 
  $$\Delta\geq\eta-\frac{D}{8}\geq\eps-\frac{D}{8}.$$ Consequently, for any
  $a\in\I_5\cap Y$, we have $a>p_0$. By the definition of the set~$A$ and the
  point $p_0$, any edge $(a,b)$ then cannot cross $(x,y)$, which is analogous to
  having $A=\emptyset$ for Case ($2$).
  \end{case}
\end{proof}

\subsection{Double Crossings}
Now, we undertake the task of finding a suitable isometry/alignment when there
is a double crossing in $\C$. In this case, we may have to consider flipping~$Y$
to construct such an isometry. We always flip $Y$ about the midpoint of $x$
and~$x'$. After flipping, the image of $Y$ is denoted by $\widetilde{Y}$ and the
image of any $b\in Y$ by~$\widetilde{b}$. We first present two technical lemmas.
\begin{lemma}\label{lem:cross}
  Let $(p,q)\in\C$ be a double crossing; see \figref{cross}. If we denote
  $h=(x-x')$, $\eps_1=(p-x)$, and $\eps_2=(y'-q)$, then we have the following:
  \begin{enumerate}[i)]
  \item $\eps_1-\eps_2\geq h$,
  \item $\eps_1-\eps_2\leq D-h$,
  \item $h\leq\frac{D}{2}$,
   and
  \item $\mod{p-\widetilde q}\leq\frac{D}{2}-h$, where $\widetilde{q}$ denotes
  the reflection of $q$ about the midpoint of $x$ and $x'$.
\end{enumerate}
\end{lemma}
\begin{proof}
  \begin{figure}[thb]
    \centering
    \includegraphics[scale=1]{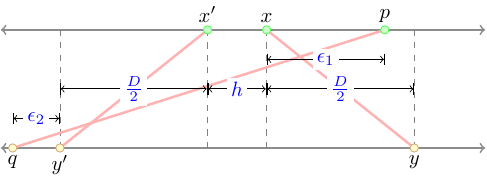}
    \caption[Double Crossing]{A double crossing $(p,q)$ is shown.}
    \label{fig:cross}
  \end{figure}
  \begin{enumerate}[i)]
  \item Let us assume the contrary, i.e., $\eps_1<\eps_2+h$. Then, the
    distortion for the pairs $(x,y)$ and $(p,q)$ becomes
    $$\mod{\eps_2+D+h-\eps_1}=\eps_2+h+D-\eps_1> D.$$ This contradicts the fact
    that the distortion of $\C$ is $D$. Therefore, we conclude that
    $\eps_1-\eps_2\geq h$.

  \item Since from (i) we have $\eps_1\geq\eps_2$, from the distortion for the
    pairs $(p,q)$ and $(x',y')$, we have
    \begin{equation}\label{eqn:gh-1}
      h+\eps_1-\eps_2\leq D.
    \end{equation}
    So, $\eps_1-\eps_2\leq D-h$.

  \item From (ii) we have $\eps_2+D\geq\eps_1$. Hence, the distortion for the
    pairs $(p,q)$ and $(x,y)$ implies 
    \begin{equation}\label{eqn:gh-2}
      \eps_2+D+h-\eps_1\leq D.
    \end{equation}
    Adding \eqnref{gh-1} and \eqnref{gh-2}, we get $2h\leq D$. Hence,
    $h\leq\frac{D}{2}$.
      
    \item If $p>\widetilde q$, then
    $$p-\widetilde
    q=\eps_1-\frac{D}{2}-\eps_2\leq(D-h)-\frac{D}{2}-\eps_2\leq\frac{D}{2}-h.$$
    Otherwise,  
    $$\widetilde
    q-p=\frac{D}{2}+\eps_2-\eps_1\leq\frac{D}{2}-(\eps_1-\eps_2)\leq\frac{D}{2}
    -h.$$
    Therefore, $\mod{p-\widetilde q}\leq\frac{D}{2}-h$.
  \end{enumerate}
\end{proof}

In our pursuit of constructing the right isometry, we first define a wide
(double) crossing. We show in \thmref{wide-cross}, that we need to flip $Y$ in
the presence of such a wide crossing.

\begin{definition}[Wide Crossing]\label{def:wide-cross} A double crossing edge
$(p,q)\in\C$ is called a \bemph{wide crossing} if either 
\begin{equation}\label{eqn:wide-p}
  p<\begin{cases}
  \min A',&\text{ if }A'\neq\emptyset\\
  x'-\frac{D}{2},&\text{ if }A\neq\emptyset\text{ and }\max A>y+\frac{3D}{4}\\
  x'-D,&\text{ otherwise}
  \end{cases}
\end{equation}
or
\begin{equation*}
  p>\begin{cases}
  \max A,&\text{ if }A\neq\emptyset\\
  x+\frac{D}{2},&\text{ if }A'\neq\emptyset\text{ and }\min A<y'-\frac{3D}{4}\\
  x+D,&\text{ otherwise}.
  \end{cases}
\end{equation*}
\end{definition}
  
Before presenting \thmref{wide-cross}, we make an important observation
first in the following technical lemma.
\begin{lemma}[Wide Crossing]\label{lem:wide-cross} Let there be a wide crossing
  $(p,q)\in\C$ and an edge $(p_0,q_0)\in\C$ such that $p_0\in A$ and $y'\leq
  q_0\leq y$. If we denote $\eps=p_0-x-D$, $\eps'=y-q_0$ and $h=x-x'$, then we
  have 
  \begin{equation}\label{eqn:eps'}
    \eps'\geq 2h+\eps.
  \end{equation}
  \end{lemma}
  \begin{proof}
  Since $A\neq\emptyset$, we have from \lemref{AB} that $A'=\emptyset$. So, for
  the wide edge $(p,q)$, we must have either $p<x'$ or $p>\max A$. We,
  therefore, consider the following two possible positions of $p$.
  \begin{figure}[thb]
      \centering
      \includegraphics[scale=1]{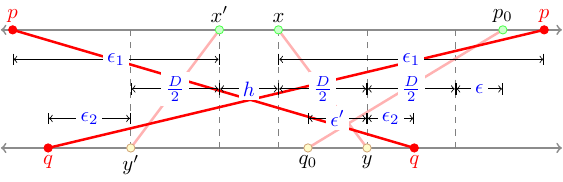}
      \caption[Wide Crossing]
      {A wide crossing $(p,q)$ is shown along with the edge $(p_0,q_0)$. Case
      ($1$) and Case ($2$) are shown in bright red.}
      \label{fig:wide-cross}
  \end{figure}
  \begin{case}[$p<x'$]
    From the distortion of the pair $(p,q)$ and $(p_0,q_0)$, we have
    \begin{align*}
    D&\geq\mod{(p_0-p)-(q-q_0)} \\
    &=\mod{(\eps_1+h+D+\eps)-(\eps'+\eps_2)} \\
    &=(\eps_1+h+D+\eps)-(\eps'+\eps_2),\\
    &\quad\quad\quad\text{ since }\eps_1\geq\eps_2
    \text{ from \lemref{cross} and }\eps'\leq D\text{ from \eqnref{eps}} \\
    &=\eps_1-\eps_2 + h + \eps + D-\eps' \\
    &\geq 2h+\eps + D -\eps',\text{ since }\eps_1\geq\eps_2+h
    \text{ from \lemref{cross}}.
    \end{align*}
  So, $\eps'\geq2h+\eps$.
  \end{case} 
    
  \begin{case}[$p>p_0$]
  From the distortion of the pair $(p,q)$ and $(p_0,q_0)$, we get
  \begin{align*}
  D&\geq\mod{(q_0-q)-(p-p_0)} \\ 
  &=\mod{(\eps_2+D+h-\eps')-(\eps_1-D-\eps)} \\
  &=\mod{(D+\eps_2-\eps_1) + (h + \eps)+(D-\eps')} \\
  &=(D+\eps_2-\eps_1) + (h + \eps)+(D-\eps'),\\
  &\quad\quad\quad\text{ since }D\geq\eps'\text{ from \eqnref{eps}}
  \text{ and }D+\eps_2-\eps_1\geq0\text{ from \lemref{cross}} \\
  &\geq2h + \eps + (D -\eps'),
  \text{ since }D+\eps_2-\eps_1\geq h\text{ from \lemref{cross}}. \\
  \end{align*}
  So, $\eps'\geq2h+\eps$. 
  \end{case}
\end{proof}

\begin{theorem}[Wide Crossing]\label{thm:wide-cross} Let $\C$ be a
  correspondence between two compact sets $X,Y\subseteq\R$ with distortion $D$.
  If there is a wide crossing $(p,q)\in\C$, then there exists a value
  $\Delta\in\R$ such that
  $$d_{H}(X,\widetilde{Y}+\Delta)\leq\frac{5D}{8},$$ where $\widetilde{Y}$
  denotes the reflection of $Y$ about the midpoint of $x$ and $x'$.
\end{theorem}
\begin{proof}
  We consider the subsets $A,A'$ of $X$ as defined in \eqnref{A}. As already
  argued in \lemref{AB}, the subsets $A$ and $A'$ cannot be both non-empty.
  Without any loss of generality, it then suffices to study only the following
  two unique cases:
  \begin{enumerate}
  \item $A\neq\emptyset,A'=\emptyset$,
  \item $A=\emptyset,A'=\emptyset$.
  \end{enumerate}
  For each of the cases, we show that a number $\Delta$ can be chosen such that
  $d_H(X,\widetilde{Y}+\Delta)\leq\frac{5D}{8}$.

\begin{case}[$A\neq\emptyset,A'=\emptyset$] 
  We denote $p_0=\max{A}$ and $\eps=p_0-x-D$. We also let $q_0\in Y\cap[y',y]$
  such that $(p_0,q_0)\in\C$, and $\eps'=(y-q_0)$; see \figref{ub-1}. As already
  shown in \eqnref{eps}, we still have 
  \begin{equation}\label{eqn:eps-wide}
    \eps\leq\eps'\leq D-\eps,\text{ and }\eps\leq\frac{D}{2}.
  \end{equation} 
  \begin{figure}[thb]
  \centering
  \includegraphics[scale=0.75]{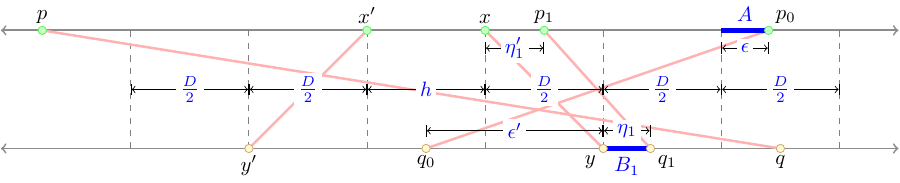}
  \includegraphics[scale=0.75]{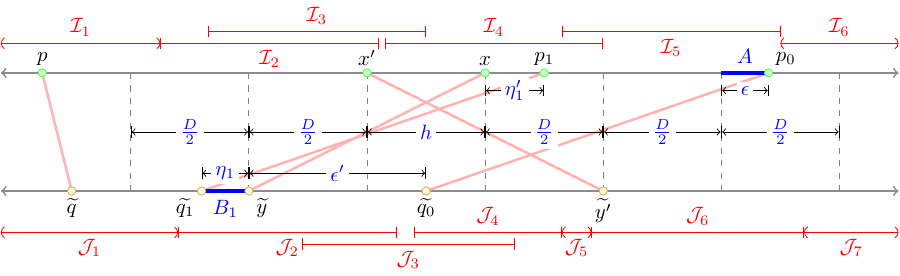}
  \caption[Double Crossing Case $1$]
  {The correspondence with a wide crossing is shown in the top. The sets $A$ and
    $B_1$ are subsets of the thick, dark blue regions on top and bottom
    respectively. In the bottom, we show the configuration when $Y$ is flipped
    about the midpoint of $x$ and $x'$.}
  \label{fig:ub-1}
  \end{figure}
  We now define the following subset of $Y$:
  \begin{equation}\label{eqn:B1}
    B_1=\{b\in Y\cap[y,\infty)\mid\exists~a\in
    X\cap[x,\infty)\mbox{ with }(a,b)\in\C\}.  
  \end{equation}
  Let $q_1=\max B_1$,
  $\eta_1=q_1-y$, and let there exist an edge $(p_1,q_1)\in\C$ with
  $\eta_1'=p_1-x$. We note from \lemref{eta-1} that
  \begin{equation}\label{eqn:eta-1}
    \eta_1\leq\eta_1'\leq D\text{ and }\eta_1\leq\frac{D}{2}-h.
  \end{equation}
  
  If we already have $d_H(X,\widetilde{Y})\leq\frac{5D}{8}$, we take $\Delta=0$.
  Let us, therefore, assume the non-trivial case that
  $d_H(X,\widetilde{Y})>\frac{5D}{8}$. As a result, we must have
  $$\max\{\eps,\eta_1\}>\frac{D}{8}.$$ This is because all points of $X$ outside
  $A$ and points of $\widetilde{Y}$ outside $B_1$ are already covered, as we
  note
  \begin{enumerate}[i)]
    \item $\mod{a-\widetilde{b}}\leq\frac{D}{2}$ for a double crossing edge $(a,b)$ from
    \lemref{cross},
    \item $\mod{a-\widetilde{q_0}}\leq\frac{D}{2}$ for any $a\in[x',x]$ due to
    the fact that $\eps'\geq2h+\eps$ from \lemref{wide-cross}, and
    \item $\mod{b-x'}\leq\frac{D}{2}$ or $\mod{b-x}\leq\frac{D}{2}$ for any
    $b\in[y',y]$ due to the fact that $h\leq\frac{D}{2}$ from \lemref{cross}.
  \end{enumerate}
  
  We now show that $d_H(X,\widetilde{Y}+\Delta)\leq\frac{5D}{8}$ for some
  translation amount $\Delta$ such that
  \begin{equation}\label{eqn:delta-1-wide}
    \begin{cases}
    \Delta=&\max\{\eps,\eta_1\}-\frac{D}{8},\text{ when }\eps\leq
    \eta_1+\frac{D}{4}  \\
    \Delta\in&\left[\eps-\frac{D}{8},\frac{5D}{8}+h+\eps-\eps'\right],
    \text{ when }
    \eta_1+\frac{D}{4}<\eps\leq\frac{D}{2}.
    \end{cases}
  \end{equation}
  While $\overrightarrow{d}_H(X,\widetilde{Y}+\Delta)\leq\frac{5D}{8}$ for any
  $\Delta$ satisfying \eqnref{delta-1-wide},  we may require a more specific
  choice of $\Delta$ in order to show
  $\overrightarrow{d}_H(\widetilde{Y}+\Delta,X)\leq\frac{5D}{8}$. 
  
  As already argued $\max\{\eps,\eta_1\}>\frac{D}{8}$, we have $\Delta>0$. We
  also note from \eqnref{eps-wide}, \eqnref{eta-1} that
  $\eps,\eta_1\leq\frac{D}{2}$ and from \eqnref{eps'} that the upper bound of
  $\Delta$:
  $$
  \frac{5D}{8}+h+\eps-\eps'\leq\frac{5D}{8}+h+\eps-(2h+\eps)\leq\frac{5D}{8}-h
  \leq\frac{5D}{8}.
  $$ 
  As a result, we note for later that $$0<\Delta\leq\frac{5D}{8}.$$ 

  In order to show
  $\overrightarrow{d}_H(X,\widetilde{Y}+\Delta)\leq\frac{5D}{8}$ for any
  $\Delta$ satisfying \eqnref{delta-1-wide}, we consider the following (possibly
  overlapping) intervals to cover the real line:
  \begin{align*}
  \I_1&=\left(-\infty,\widetilde{q_1}-\frac{5D}{8}+\Delta\right),
  \I_2=\left[\widetilde{q_1}-\frac{5D}{8}+\Delta,
  \widetilde{q_1}+\frac{5D}{8}+\Delta\right], \\
  \I_3&=\left[\widetilde{y}-\frac{5D}{8}+\Delta,
    \widetilde{y}+\frac{5D}{8}+\Delta\right],
  \I_4=\left[\widetilde{q_0}-\frac{5D}{8}+\Delta,\widetilde{q_0}+\frac{5D}{8}+
  \Delta\right], \\
  \I_5&=\left[\widetilde{y'}-\frac{5D}{8}+\Delta,\widetilde{y'}+\frac{5D}{8}+
  \Delta\right],\mbox{ and } 
  \I_6=\left(\widetilde{y'}+\frac{5D}{8}+\Delta,\infty\right).
  \end{align*}
  Since $\widetilde{y}-\widetilde{q_1}=\eta_1\leq D$ from \eqnref{eta-1}, the
  intervals $\I_2$ and $\I_3$ intersect. Since
  $\widetilde{q_0}-\widetilde{y}=\eps'\leq D$ from \eqnref{eps-wide}, the
  intervals $\I_3$ and $\I_4$ intersect. Since $\eps'\geq h$ from
  \lemref{wide-cross}, we have $\widetilde{y'}-\widetilde{q_0}=(D+h)-\eps'\leq
  D$. So, the intervals $\I_4$ and $\I_5$ intersect. As a result, the union of
  the intervals is the entire real line. For an arbitrary point $a\in X$ from
  any of the above intervals, we show that there exists a point $b\in Y$ such
  that $\mod{a-(\widetilde{b}+\Delta)}\leq\frac{5D}{8}$. The intervals $\I_2$,
  $\I_3$, $\I_4$, and $\I_5$ are covered by $q_1$, $y$, $q_0$, and $y'$,
  respectively. So, we present our argument only for $\I_1$ and $\I_6$.

  \paragraph{($\pmb{\I_1}$)} Let $a\in\I_1\cap X$ with an edge $(a,b)\in\C$. As
  we note now, $a$ satisfies the condition \eqnref{wide-1} of
  \lemref{wide-no-cross}. When $\eps\leq\eta_1+\frac{D}{4}$, we have 
  \begin{align*}
    x'-a&>\frac{D}{2}+\eta_1+\left(\frac{5D}{8}-\Delta\right)
    =D+\eta_1-\Delta+\frac{D}{8} \\
    &=D+\eta_1-\left(\max\{\eps,\eta_1\}-\frac{D}{8}\right)
    +\frac{D}{8},\text{ from \eqnref{delta-1-wide}} \\
    &=D+\left(\eta_1-\max\{\eps,\eta_1\}\right)+\frac{D}{4}
    =\begin{cases}
      D+\frac{D}{4},\text{ if }\eta_1\geq\eps \\
      D+\frac{D}{4}+\eta_1-\eps\text{ if }\eta_1\leq\eps
    \end{cases}\\
    &\geq D,\text{ since }\eps\leq\eta_1+\frac{D}{4}.
  \end{align*}
  When $\eps>\eta_1+\frac{D}{4}$, we have 
  $$x'-a>\frac{D}{2}+\eta_1+\left(\frac{5D}{8}-\Delta\right)\geq\frac{D}{2},
  \text{ as }\Delta\leq\frac{5D}{8}.$$  Therefore, \lemref{wide-no-cross}
  implies that $(a,b)$ is a double crossing. As a result, after flipping $Y$,
  the edge $(a,\widetilde{b})$ does not cross $(p_0,\widetilde{q_0})$.
  Consequently, when $\widetilde{b}\geq a$, due to the distortion bound of $D$
  of the edges $(a,\widetilde{b})$ and $(p_0,\widetilde{q_0})$, we must have 
  $$
    D\geq(p_0-a)-(\widetilde{q_0}-\widetilde{b})
    =(\widetilde{b}-a)+(p_0-\widetilde{q_0}).
  $$
  So, 
  \begin{align*}
    \Delta+(\widetilde{b}-a)&\leq\Delta + [D-(p_0-\widetilde{q_0})] \\
    &=\Delta+D-\left[\left(\eps+\frac{D}{2}\right)+(D+h-\eps')\right]\\
    &=\Delta+\eps'-\eps-h-\frac{D}{2} \\
    &\leq\left(\frac{5D}{8}+h+\eps-\eps'\right)+\eps'-\eps-h-\frac{D}{2},
    \text{ from \eqnref{delta-1-wide}}\\
    &\leq\frac{5D}{8}.
  \end{align*}
  Also, when $\widetilde{b}\leq a$, we must have
  $\mod{a-(\widetilde{b}+\Delta)}\leq\frac{5D}{8}$, since
  $\Delta\leq\frac{5D}{8}$. Therefore, in either case, we have
  $\mod{a-(\widetilde{b}+\Delta)}\leq\frac{5D}{8}$.
  
  \paragraph{($\pmb{\I_6}$)} Let $a\in\I_6\cap X$ and $(a,b)\in\C$ an edge. We
  note that $a>p_0$, as we have $\Delta\geq\eps-\frac{D}{8}$ from
  \eqnref{delta-1-wide}. \lemref{wide-no-cross} implies that $(a,b)$ is a double
  crossing edge. Therefore, after flipping $Y$, the edge $(a,\widetilde{b})$
  cannot cross $(p_0,\widetilde{q_0})$.  The rest of the argument presented for
  $\I_1$ then goes through. Hence arguing similar to $\I_1$, we have
  $\mod{a-(\widetilde{b}+\Delta)}\leq\frac{5D}{8}$.

  In order to show that $\overrightarrow{d}_H(\widetilde{Y}+\Delta,X)
  \leq\frac{5D}{8}$, we define the following intervals:
  \begin{align*}
  \J_1 &= \left(-\infty,x'-\frac{5D}{8}-\Delta\right),
  \J_2 = \left[x'-\frac{5D}{8}-\Delta,x'+\frac{5D}{8}-\Delta\right],\\
  \J_3 &= \left[x-\frac{5D}{8}-\Delta,x+\frac{5D}{8}-\Delta\right],
  \J_4 = \left[p_1-\frac{5D}{8}-\Delta,p_1+\frac{5D}{8}-\Delta\right],\\
  \J_5 &= \left(p_1+\frac{5D}{8}-\Delta,p_0-\frac{5D}{8}-\Delta\right),
  \J_6 = \left[p_0-\frac{5D}{8}-\Delta,p_0+\frac{5D}{8}-\Delta\right],\\
  \text{ and } & \J_7 = \left(p_0+\frac{5D}{8}-\Delta,\infty\right).
  \end{align*}
  Since $x-x'=h\leq\frac{D}{2}$ from \lemref{cross}, the intervals $\J_2$ and
  $\J_3$ intersect. Since $p_1-x=\eta_1'\leq D$ from \eqnref{eta-1}, the
  intervals $\J_3$ and $\J_4$ intersect. As a result, the union of the intervals
  is the entire real line. For an arbitrary point
  $\widetilde{b}\in\widetilde{Y}$ from any of the above intervals, we show that
  there exists a point $a\in X$ such that
  $\mod{a-(\widetilde{b}+\Delta)}\leq\frac{5D}{8}$. The intervals $\J_2$,
  $\J_3$, $\J_4$, and $\J_6$ are covered by $x'$, $x$, $p_1$, and $p_0$,
  respectively. So, we present the argument only for $\J_1,\J_5$, and $\J_7$.
  While the argument for $\J_1,\J_7$ goes through for any $\Delta$ satisfying
  \eqnref{delta-1-wide}, the interval $\J_5$ may require a more particular
  choice of $\Delta$.
  
  \paragraph{($\pmb{\J_1}$)} Let $\widetilde{b}\in\widetilde{Y}\cap\J_1$ and
  $(a,b)\in\C$ an edge. From \eqnref{delta-1-wide}, we get
  $\Delta\geq\eta_1-\frac{D}{8}$, so $\widetilde{b}<\widetilde{q_1}$, i.e,
  $b>q_1$. Since $q_1=\max B_1$, the edge $(a,b)$ has to cross $(p_0,q_0)$. So
  after flipping $Y$, the edge $(a,\widetilde{b})$ does not cross
  $(p_0,\widetilde{q_0})$.  The rest of the argument presented for the interval
  $\I_1$ then goes through. Hence arguing similar to $\I_1$, we have
  $\mod{a-(\widetilde{b}+\Delta)}\leq\frac{5D}{8}$. 
  
  \paragraph{($\pmb{\J_5}$)} We observe that $\J_5\neq\emptyset$ if and only if
  $\eps>\frac{D}{4}+\eta_1'$. There is nothing to show if $\J_5$ is empty. So,
  we assume that $\eps>\frac{D}{4}+\eta_1'$, and claim that there must exist a
  number $\Delta\in\left[\eps-\frac{D}{8},\frac{5D}{8}+h+\eps-\eps'\right]$ such
  that for any $\widetilde{b}\in\widetilde{Y}\cap\J_5$ there exists $a\in X$
  with $\mod{a-(\widetilde{b}+\Delta)}\leq\frac{5D}{8}$. We also observe that
  the choice of $\Delta$ here satisfies \eqnref{delta-1-wide}, because
  $\eps>\frac{D}{4}+\eta_1'$ implies $\eps>\frac{D}{4}+\eta_1$ due to
  \eqnref{eta-1}.

  We argue by contradiction. Let us assume that no such $\Delta$ exists in the
  given interval, i.e., for all
  $\Delta\in\left[\eps-\frac{D}{8},\frac{5D}{8}+h+\eps-\eps'\right]$, the
  following subset of $\J_5$ is non-empty:
  $$
    E(\Delta)=\left\{\widetilde{b}\in\widetilde{Y}\cap\J_5\mid X\cap\left[\widetilde{b}-\frac{5D}{8}+\Delta,\widetilde{b}+\frac{5D}{8}+\Delta\right]=\emptyset\right\}.
  $$
  Let us also define
  $$
  \widetilde{Y}_\leq=\left\{b\in Y\mid\text{there exists an edge }(a,b)\in\C
  \text{ with }a\leq\widetilde{b}\right\}
  $$
  and
  $$
  \widetilde{Y}_\geq=\left\{b\in Y\mid\text{there exists an edge }(a,b)\in\C
  \text{ with }a\geq\widetilde{b}\right\}.
  $$
  We note from \lemref{LR} that $E(\Delta)\cap\widetilde{Y}_\leq$ and
  $E(\Delta)\cap\widetilde{Y}_\geq$ cannot be both non-empty for any given
  $\Delta$. Therefore, $E(\Delta)\neq\emptyset$ for all $\Delta$ implies
  \begin{enumerate}[i)]
    \item either $E(\Delta)\cap\widetilde{Y}_\leq$ is empty and
    $E(\Delta)\cap\widetilde{Y}_\geq$ non-empty for all $\Delta$, or
    \item $E(\Delta)\cap\widetilde{Y}_\leq$ is non-empty and
    $E(\Delta)\cap\widetilde{Y}_\geq$ empty for all $\Delta$.
  \end{enumerate}
  In order to arrive at a contradiction, we now show, however, that
  $E(\Delta)\cap\widetilde{Y}_\geq$ is empty when $\Delta=\eps-\frac{D}{8}$,
  whereas $E(\Delta)\cap\widetilde{Y}_\leq$ is empty when
  $\Delta=\frac{5D}{8}+h+\eps-\eps'$.
  
  When $\Delta=\eps-\frac{D}{8}$, we show that
  $E(\Delta)\cap\widetilde{Y}_\geq=\emptyset$. Consider any
  $\widetilde{b}\in\J_5\cap\widetilde{Y}_\geq$ with an edge $(a,b)\in\C$ such
  that $a\geq\widetilde{b}$. Due to the distortion bound of $D$ for the pair of
  edges $(x',\widetilde{y'})$ and $(a,\widetilde{b})$, we have $(a-x') \leq D +
  (\widetilde{y'}-\widetilde{b})$. So,
  \begin{align*} 
    a-(\widetilde{b}+\Delta)
    &\leq (x' + D + \widetilde{y'} - \widetilde{b} )- (\widetilde{b}+\Delta) \\
    &\leq x' + D + \widetilde{y'} -2\widetilde{b}-\Delta \\
    &< x' + D + \widetilde{y'}-2\left(p_1+\frac{5D}{8}-\Delta\right)-\Delta,\text{ since }\widetilde{b}\in\J_5 \\
    &\leq x + D + \widetilde{y'}-2\left(x+\frac{5D}{8}-\Delta\right)-\Delta,\text{ since }p_1\geq x\geq x' \\
    &=(\widetilde{y'} - x) -\frac{D}{4}+\Delta \\
    &=\frac{D}{2}-\frac{D}{4}+\Delta \\
    &=\frac{D}{4}+\left(\eps-\frac{D}{8}\right) \\
    &=\eps+\frac{D}{8}\leq\frac{5D}{8},
    \text{ since }\eps\leq\frac{D}{2}\text{ from \eqnref{eps-wide}}.
  \end{align*}
  So, $\widetilde{b}\notin E(\Delta)$. Therefore,
  $E(\Delta)\cap\widetilde{Y}_\geq=\emptyset$ for $\Delta=\eps-\frac{D}{8}$. 
  
  We now show that $E(\Delta)\cap\widetilde{Y}_\leq=\emptyset$ when
  $\Delta=\frac{5D}{8}+h+\eps-\eps'$. To see this, consider
  $\widetilde{b}\in\widetilde{Y}\cap\J_5$ and an edge $(a,\widetilde{b})\in\C$
  such that $a\leq\widetilde{b}$. Due to the distortion bound of $D$ for the
  pair of edges $(p_0,\widetilde{q_0})$ and $(a,\widetilde{b})$, we have
  $(p_0-a) \leq D + (\widetilde{b}-\widetilde{q_0})$. So,
  \begin{align*}
    (\widetilde{b}+\Delta)-a 
    &\leq(\widetilde{b}+\Delta) + D + (\widetilde{b}-\widetilde{q_0}) - p_0\\    
    &\leq\Delta + D + 2\widetilde{b}-\widetilde{q_0}- p_0\\
    &<\Delta+D+2\left(p_0-\frac{5D}{8}-\Delta\right)-\widetilde{q_0}-p_0,\text{ since }\widetilde{b}\in\J_5  \\
    &=(p_0-\widetilde{q_0})-\frac{D}{4}-\Delta \\
    &=\left(D+h-\eps'+\frac{D}{2}+\eps\right)-\frac{D}{4}-\Delta \\
    &=\left(D+h-\eps'+\frac{D}{2}+\eps\right)-\frac{D}{4}-\left(\frac{5D}{8}+h+\eps-\eps'\right) \\
    &=\frac{5D}{8}.
  \end{align*}
  So, $\widetilde{b}\notin E(\Delta)$. Therefore,
  $E(\Delta)\cap\widetilde{Y}_\leq=\emptyset$ for
  $\Delta=\frac{5D}{8}+h+\eps-\eps'$. This is a contradiction. Hence, there must
  exist a number $\Delta$ satisfying \eqnref{delta-1-wide} as claimed.
  
  \paragraph{($\pmb{\J_7}$)} Let $\widetilde{b}\in\widetilde{Y}\cap\J_7$ and
  $(a,b)\in\C$ an edge. Since $\Delta\leq\frac{5D}{8}$, we have
  $\widetilde{b}>p_0\geq y+\frac{D}{2}$, i.e., $b<y'-\frac{D}{2}$. So,
  \lemref{wide-no-cross-1} implies that $(a,b)$ is a double crossing, i.e.,
  $a>x$. We further argue that we must also have $a>p_0$.

  If not, i.e., $x<a\leq p_0$, then $(p_0,\widetilde{q_0})$ crosses
  $(a,\widetilde{b})$. Consequently, the distortion of the pair 
  \begin{align*}
  &(\widetilde{b}-\widetilde{q_0})-(p_0-a)\\
  &\geq(\widetilde{b}-\widetilde{q_0})-(\widetilde{b}-a),
  \text{ as }\widetilde{b}\in\J_7,\text{ we have }p_0\leq\widetilde{b} \\
  &=(\widetilde{b}-p_0)+(p_0-\widetilde{y'})
  +(\widetilde{y'}-\widetilde{q_0})-(\widetilde{b}-a)\\
  &\geq(\widetilde{b}-p_0)+(p_0-\widetilde{y'})
  +(\widetilde{y'}-\widetilde{q_0})-\frac{D}{2},\text{ from \lemref{cross}}\\
  &=(\widetilde{b}-p_0)+\left(\frac{D}{2}+\eps\right)+
  \left(D+h-\eps'\right)-\frac{D}{2} \\
  &>\left(\frac{5D}{8}-\Delta\right)+\left(\frac{D}{2}+\eps\right)+
  \left(D+h-\eps'\right)-\frac{D}{2},\text{ as }\widetilde{b}\in\J_7 \\
  &=\frac{13D}{8}+\eps-\eps'+h-\Delta\\
  &\geq\frac{13D}{8}+\eps-\eps'+h-\left(\frac{5D}{8}+\eps-\eps'+h\right),
  \text{ from \eqnref{delta-1-wide}}\\
  &=D
  \end{align*}
  This is a contradiction, so $a>p_0$. Therefore, the edge $(a,\widetilde{b})$
  cannot cross $(p_0,\widetilde{q_0})$. Hence arguing similar to $\I_1$, we have
  $\mod{a-(\widetilde{b}+\Delta)}\leq\frac{5D}{8}$. 
  \end{case}

  \begin{case}[$A=\emptyset,A'=\emptyset$]
  We define 
  \begin{equation}\label{eqn:B2}
    B_2=\{b\in Y\cap(-\infty,y']\mid\exists~a\in
    X\cap(-\infty,x']\mbox{ with }(a,b)\in\C\}.
  \end{equation} 
  Let $q_2=\min B_2$, $\eta_2=y'-q_2$, and let there exist an edge
  $(p_2,q_2)\in\C$ with $\eta_2'=x'-p_2'$; see \figref{ub1}. Following the same
  argument presented for \eqnref{eta-1}, we have
  \begin{equation}\label{eqn:eta-2}
    \eta_2\leq\eta_2'\leq D\text{ and }\eta_2\leq\frac{D}{2}-h.
  \end{equation}
  If we already have $d_H(X,\widetilde{Y})\leq\frac{5D}{8}$, we take $\Delta=0$.
  Let us, therefore, assume the non-trivial case that
  $d_H(X,\widetilde{Y})>\frac{5D}{8}$. We have observed in \lemref{cross} that
  for a double crossing edge $(a,b)$, we already have
  $\mod{a-\widetilde{b}}\leq\frac{D}{2}$, after flipping $Y$. This together with
  the fact that $A,A'=\emptyset$ implies we must have either
  \begin{enumerate}[i)]
    \item $a_0\in\left(x'+\frac{D}{8},x-\frac{D}{8}\right)\cap X$ such that
    $\min_{\widetilde{y}\in\widetilde{Y}}\mod{a_0-\widetilde{y}}>\frac{5D}{8}$,
    or
    \item $b_0\in B_1\cup B_2$ such that $\min_{x\in
    X}\mod{\widetilde{b_0}-x}>\frac{5D}{8}$.
  \end{enumerate} 

  When $h\leq\frac{3D}{8}$, such $a_0$ cannot exist.  Because, we have that
  $\mod{a_0-\widetilde{b}}\leq\frac{5D}{8}$ for any edge $(a_0,b)$, due to the
  fact that $\mod{a-b}\leq\frac{D}{2}$. Consequently, $b_0$ must exist implying
  that either $\eta_1>\frac{D}{8}$ or $\eta_2>\frac{D}{8}$. Without any loss of
  generality, we assume $\eta_1\geq\eta_2>\frac{D}{8}$ so that the direction of
  the translation of $\widetilde{Y}$ is from left to right. 

  When $h>\frac{3D}{8}$, on the other hand, from \eqnref{eta-1},\eqnref{eta-2}
  we have $\eta_1,\eta_2\leq\frac{D}{8}$. Consequently, $b_0$ cannot exist, so
  $a_0$ must exist. We also note from the existence of such $a_0$ that the set
  \begin{equation}\label{eqn:H}
    H:=\left(y'+h-\frac{D}{4},y-h+\frac{D}{4}\right)\cap Y=\emptyset.
  \end{equation}
  The set $H$ is indicated in \figref{ub1}. This is because for any $b\in H$ and
  $a\in\left(x'+\frac{D}{8},x-\frac{D}{8}\right)$, we have
  $$\mod{a-\widetilde{b}}
  \leq\mod{\left(x-\frac{D}{8}\right)-\left(y'+h-\frac{D}{4}\right)}
  =\frac{5D}{8}.$$ Furthermore, $a_0$ belongs to either
  $C_L:=\left(x'+\frac{D}{8},x'+h-\frac{D}{4}\right)\cap X$ or (its mirror
  image) $C_R:=\left(x-h+\frac{D}{4},x-\frac{D}{8}\right)\cap~X$. This is
  because for any $a\in\left[x'+h-\frac{D}{4},x-h+\frac{D}{4}\right]$ with edge
  $(a,b)$, we would have $b\in H$, due to the fact that
  $\mod{a-b}\leq\frac{D}{2}$. The sets $C_L$ and $C_R$ are also depicted in
  \figref{ub1}. Without any loss of generality, we then assume that 
  \begin{equation}\label{eqn:a_0}
    a_0\in\left(x'+\frac{D}{8},x'+h-\frac{D}{4}\right)\cap X.
  \end{equation}
  The choice is justified, because one can alternatively consider $-X,-Y$
  without altering the distortion of the correspondence. The assumption again
  retains the left-to-right direction of the translation of $\widetilde{Y}$.

  Now, we take our
  \begin{equation}\label{eqn:delta-2-wide}
    \Delta=\begin{cases}
      &\eta_1-\eta_2-\frac{D}{8},
      \text{ when }h\leq\frac{3D}{8} \\
      &h-\frac{3D}{8},\text{ when }\frac{3D}{8}<h\leq\frac{D}{2}.
    \end{cases}
  \end{equation}
  We immediately note that $0<\Delta\leq\frac{D}{2}$, as
  $\eta_1\leq\frac{D}{2}$.
  \begin{figure}[thb]
  \centering
  \includegraphics[scale=0.8]{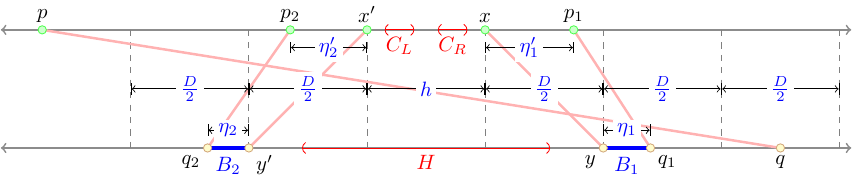}
  \includegraphics[scale=0.8]{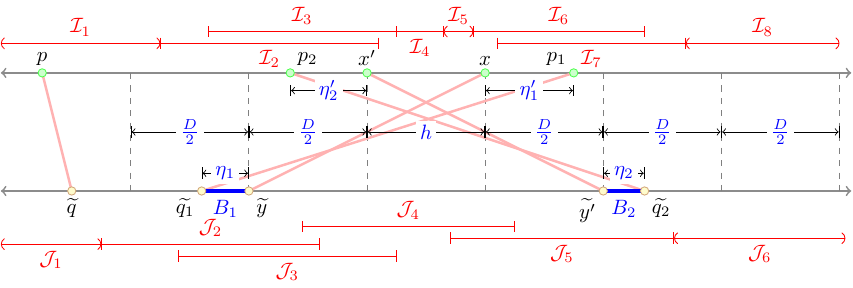}
  \caption[Wide Crossing Case $2$]{Wide crossing exists, and both 
  $A=\emptyset$, $A'=\emptyset$}
  \label{fig:ub1}
  \end{figure}

  In order to show that $\overrightarrow{d}_H(X,\widetilde{Y}+\Delta)
  \leq\frac{5D}{8}$, we define the following intervals:
  \begin{align*}
  \I_1 &= \left(-\infty,\widetilde{q_1}-\frac{5D}{8}+\Delta\right),
  \I_2 = \left[\widetilde{q_1}-\frac{5D}{8}+\Delta,\widetilde{q_1}+\frac{5D}{8}+\Delta\right], \\
  \I_3 &= \left[\widetilde{y}-\frac{5D}{8}+\Delta,\widetilde{y}+\frac{5D}{8}+\Delta\right],
  \I_4 = \left[x'+\frac{D}{8}+\Delta,x-\frac{h}{2}+\frac{D}{16}+\frac{\Delta}{2}\right], \\
  \I_5 &= \left(x-\frac{h}{2}+\frac{D}{16}+\frac{\Delta}{2},x-\frac{D}{8}+\Delta\right),
  \I_6 = \left[\widetilde{y'}-\frac{5D}{8}+\Delta,\widetilde{y'}+\frac{5D}{8}+\Delta\right], \\
  \I_7 &= \left[\widetilde{q_2}-\frac{5D}{8}+\Delta,\widetilde{q_2}+\frac{5D}{8}+\Delta\right],
  \mbox{ and }
  \I_8= \left(\widetilde{q_2}+\frac{5D}{8}+\Delta,\infty\right).
  \end{align*}
  Since $\widetilde{y}-\widetilde{q_1}=\eta_1\leq D$ from \eqnref{eta-1}, the
  intervals $\I_2$ and $\I_3$ intersect. Since
  $\widetilde{q_2}-\widetilde{y'}=\eta_2\leq D$ from \eqnref{eta-2}, the
  intervals $\I_6$ and $\I_7$ intersect. As a result, the union of the intervals
  is the entire real line. For an arbitrary point $a\in X$ from any of the above
  intervals, we show that there exists a point $b\in Y$ such that
  $\mod{a-(\widetilde{b}+\Delta)}\leq\frac{5D}{8}$. The intervals $\I_2$,
  $\I_3$, $\I_6$, and $\I_7$ are covered by $q_1$, $y$, $y'$, and $q_2$,
  respectively. So, we present the argument only for $\I_1,\I_4$, $\I_5$, and
  $\I_8$.

  \paragraph{$\pmb{(\I_1)}$} Let $a\in\I_1\cap X$ and $(a,b)\in\C$ an edge. We
  argue that $b>q_1$. We first observe that 
  \begin{align*}
  x'-a&=(x'-\widetilde{q_1})+(\widetilde{q_1}-a)
  =\left(\frac{D}{2}+\eta_1\right)+(\widetilde{q_1}-a) \\
  &>\left(\frac{D}{2}+\eta_1\right)
  +\left(\frac{5D}{8}-\Delta\right)
  =D+\left(\frac{D}{8}+\eta_1-\Delta\right) \\
  &\geq D,\text{ as last term is non-negative for any }\Delta
  \text{ satisfying \eqnref{delta-2-wide}}.
  \end{align*}
  So, \lemref{wide-no-cross} implies that $b>y$. If we assume $b\in[y,q_1]$,
  then the distortion of the edges $(a,b)$ and $(p_1,q_1)$ exceeds $D$: 
  $$(p_1-a)-(q_1-b)>(D+h+\eta_1')-\eta_1\geq D\text{ as }\eta_1'\geq\eta_1
  ,\text{ from \eqnref{eta-1}}.$$ Therefore, $b>q_1$ as claimed. As a result,
  after flipping $Y$ the edge $(a,\widetilde{b})$ does not cross
  $(p_1,\widetilde{q_1})$. Consequently, when $\widetilde{b}\geq a$, due to the
  distortion bound of $(a,b)$ and $(p_1,\widetilde{q_1})$, we must have 
  $$(\widetilde{b}-a)+(p_1-\widetilde{q_1})\leq D.$$
  So,
  \begin{align*}
  \Delta+(\widetilde{b}-a)
  &\leq\Delta + D-(p_1-\widetilde{q_1}) \\
  &=\Delta + D-\left(\eta_1+\frac{D}{2}+h+\eta_1'\right) \\
  &\leq\Delta + \frac{D}{2}-\eta_1-h.
  \end{align*}
  When $h\leq\frac{3D}{8}$, from \eqnref{delta-2-wide} we have 
  $$\mod{a-(\widetilde{b}+\Delta)}\leq\Delta+\frac{D}{2}-\eta_1-h
  =\left(\eta_1-\eta_2-\frac{D}{8}\right)+\frac{D}{2}-\eta_1-h\leq\frac{5D}{8}.$$
  When $h>\frac{3D}{8}$, again from \eqnref{delta-2-wide} we have 
  $$\mod{a-(\widetilde{b}+\Delta)}\leq\Delta+\frac{D}{2}-\eta_1-h
  =\left(h-\frac{3D}{8}\right)+\frac{D}{2}-\eta_1-h\leq\frac{5D}{8}.$$
  When $a\geq\widetilde{b}$, from \lemref{cross}, we have
  $a-\widetilde{b}\leq\frac{D}{2}$. Since $\Delta\leq\frac{D}{2}$ as already
  noted, we must have
  $$\mod{a-(\widetilde{b}+\Delta)}\leq\frac{5D}{8}.$$
  Hence, we have $\mod{a-(\widetilde{b}+\Delta)}\leq\frac{5D}{8}$.
  
  \paragraph{$\pmb{(\I_4)}$} We note that $\I_4\neq\emptyset$ if and only if
  $h\geq\frac{D}{4}$. There is nothing to show if $\I_4$ is empty So, we assume
  that $h\geq\frac{D}{4}$. Also, the choice of $\Delta$ in \eqnref{delta-2-wide}
  satisfies \eqnref{Delta}. Therefore, \lemref{flip-1} implies that
  $\mod{a-(\widetilde{b}+\Delta)}\leq\frac{5D}{8}$.

  \paragraph{$\pmb{(\I_5)}$} We note that $\I_5\neq\emptyset$ if and only if
  $h>\frac{3D}{8}$. There is nothing to show if $\I_5$ is empty. So, we assume
  that $h>\frac{3D}{8}$ and $a\in\I_5\cap X$. From \eqnref{delta-2-wide}, we
  have $\Delta=h-\frac{3D}{8}$. As already discussed there is $a_0$ satisfying
  \eqnref{a_0}. Let $(a_0,b)$ be an edge. We argue that
  $\mod{a-(\widetilde{b}+\Delta)}\leq\frac{5D}{8}$. 
  
  We note from \lemref{standard} that $\mod{a_0-b}\leq\frac{D}{2}$.
  Consequently,
  $$b\in\left(\left[x'+\frac{D}{8}\right]-\frac{D}{2},
  \left[x'+h-\frac{D}{4}\right]+\frac{D}{2}\right)
  =\left(y'+\frac{D}{8},x'+h+\frac{D}{4}\right).$$ Hence,   
  $\widetilde{b}\in\left(x-h-\frac{D}{4},y-\frac{D}{8}\right)$. Considering the
  opposite endpoints of this interval and $\I_5$, we can write 
  \begin{align*}
    &\mod{(\widetilde{b}+\Delta)-a} \\
    &\leq\max\left\{\mod{\left(x-h-\frac{D}{4}\right)+\Delta-\left(x-\frac{D}{8}+\Delta\right)},
    \mod{\left(y-\frac{D}{8}\right)+\Delta-\left(x-\frac{h}{2}+\frac{D}{16}+\frac{\Delta}{2}\right)}\right\},\\
    &=\max\left\{\mod{\Delta-h},\mod{\frac{\Delta}{2}+\frac{h}{2}+\frac{5D}{16}}\right\}\\
    &=\max\left\{\mod{\left(h-\frac{3D}{8}\right)-h},
    \mod{\frac{1}{2}\left(h-\frac{3D}{8}\right)+\frac{h}{2}+\frac{5D}{16}}\right\}\\
    &=\max\left\{\frac{3D}{8},h+\frac{D}{8}\right\}\\
    &=\max\left\{\frac{3D}{8},\frac{5D}{8}\right\}
    \text{ since }h\leq\frac{D}{2}\text{ from \lemref{cross}}.\\
    &=\frac{5D}{8}.
  \end{align*}

  \paragraph{$\pmb{(\I_8)}$} For $a\in\I_8\cap X$, we have $a>x+D$. By
   \lemref{wide-no-cross}, we then have $b<y'$. Consequently, after flipping
   $Y$, the edge $(a,\widetilde{b})$ cannot cross $(p_1,\widetilde{q_1})$.  The
   rest of the argument presented for $\I_1$ then goes through. Hence arguing
   similar to $\I_1$, we have $\mod{a-(\widetilde{b}+\Delta)}\leq\frac{5D}{8}$. 

  In order to show that $\overrightarrow{d}_H(\widetilde{Y}+\Delta,X)
  \leq\frac{5D}{8}$, we define the following intervals:
  \begin{align*}
  \J_1 &= \left(-\infty,p_2-\frac{5D}{8}-\Delta\right),
  \J_2 = \left[p_2-\frac{5D}{8}-\Delta,p_2+\frac{5D}{8}-\Delta\right], \\
  \J_3 &= \left[x'-\frac{5D}{8}-\Delta,x'+\frac{5D}{8}-\Delta\right],
  \J_4 = \left[x-\frac{5D}{8}-\Delta,x+\frac{5D}{8}-\Delta\right], \\
  \J_5 &= \left[p_1-\frac{5D}{8}-\Delta,p_1+\frac{5D}{8}-\Delta\right],
  \J_6 = \left(p_1+\frac{5D}{8}-\Delta,\infty\right).
  \end{align*}
  Since $x'-p_2=\eta_2'\leq D$ from \eqnref{eta-2}, the intervals $\J_2$ and
  $\J_3$ intersect. Since $x-x'=h\leq D$ from \lemref{cross}, the intervals
  $\J_3$ and $\J_4$ intersect. Since $p_1-x=\eta_1'\leq D$ from \eqnref{eta-1},
  the intervals $\J_4$ and $\J_5$ intersect. As a result, the union of the
  intervals is the entire real line. For an arbitrary point
  $\widetilde{b}\in\widetilde{Y}$ from any of the above intervals, we show that
  there exists a point $a\in X$ such that
  $\mod{a-(\widetilde{b}+\Delta)}\leq\frac{5D}{8}$. We present the argument only
  for $\J_1$ and $\J_6$.

  \paragraph{$\pmb{(\J_1)}$} For any $\widetilde{b}\in\widetilde{Y}\cap\J_1$, we
  first argue that $\widetilde{b}<\widetilde{q_1}$. When $\eta_1<\frac{D}{8}$,
  this is definitely true. When $\eta_1\geq\frac{D}{8}$, then \eqnref{eta-1}
  yields $h\leq\frac{3D}{8}$. So, 
  \begin{align*}
    \widetilde{b}-\widetilde{q_1} &<\left(p_2-\frac{5D}{8}-\Delta\right)
    -\widetilde{q_1}=(x'-\eta_2')-\frac{5D}{8}-\Delta-(\widetilde{y}-\eta_1) \\
    &=(x'-\widetilde{y})-\frac{5D}{8}+(\eta_1-\eta_2'-\Delta)
    =\frac{D}{2}-\frac{5D}{8}+(\eta_1-\eta_2'-\Delta) \\
    &=\eta_1-\eta_2'-\frac{D}{8}-\Delta
    =\eta_1-\eta_2'-\frac{D}{8}-\left(\eta_1-\eta_2-\frac{D}{8}\right),
    \text{ from \eqnref{delta-2-wide}} \\
    &=\eta_2-\eta_2'\leq0,\text{ from \eqnref{eta-2}}.
  \end{align*} 
  So, $\widetilde{b}<\widetilde{q_1}$. Consequently, $b> q_1$. Since $q_1=\max
  B_1$, for any edge $(a,b)$, we must have $a<x$ by the definition of the set
  $B_1$; see \eqnref{B1}. Therefore, $(a,b)$ crosses $(p_1,q_1)$. So after
  flipping $Y$, the edge $(a,\widetilde{b})$ does not cross
  $(p_1,\widetilde{q_1})$.  The rest of the argument presented for $\I_1$ then
  goes through. Hence arguing similar to $\I_1$, we have
  $\mod{a-(\widetilde{b}+\Delta)}\leq\frac{5D}{8}$. 
  
  \paragraph{$\pmb{(\J_6)}$} For $\widetilde{b}\in\widetilde{Y}\cap\J_6$, we
  argue that $\widetilde{b}>\widetilde{q_2}$. We have 
  \begin{align*}
  \widetilde{b}-\widetilde{q_2}&>\left(p_1+\frac{5D}{8}-\Delta\right)
  -\widetilde{q_2}=(x+\eta_1')+\frac{5D}{8}-\Delta
  -(\widetilde{y'}+\eta_2)\\
  &=(x-\widetilde{y'})+\frac{5D}{8}+(\eta_1'-\eta_2-\Delta)=-\frac{D}{2}
  +\frac{5D}{8}+(\eta_1'-\eta_2-\Delta)\\
  &=\eta_1'-\eta_2+\frac{D}{8}-\Delta\geq\eta_1-\eta_2+\frac{D}{8}-\Delta,
  \text{ as }\eta_1'\geq\eta_1 \\
  &\geq\min\left\{\eta_1-\eta_2+\frac{D}{8}-\left(\eta_1-\eta_2-\frac{D}{8}\right),
  \eta_1-\eta_2+\frac{D}{8}-\left(h-\frac{3D}{8}\right)\right\} \\
  &=\min\left\{\frac{D}{4},\eta_1-\eta_2+\frac{D}{2}-h\right\} \\
  &\geq0,\text{ since }\eta_1\geq\eta_2\text{ as assumed, and }h\leq\frac{D}{2}.
  \end{align*}
  So, $\widetilde{b}>\widetilde{q_2}$. Consequently, $b<q_2$. Since $q_2=\min
  B_2$, for edge $(a,b)$, we have $a>x'$ by the definition of the set $B_2$; see
  \eqnref{B2}. On the other hand, we have $a\notin[x',x]$ due to
  \lemref{standard}. As a result, $a>x$. Furthermore, we must have $a>p_1$. If
  not, i.e., $a\in[x,p_1]$, then the distortion of the pair of edges
  $(p_1,\widetilde{q_1})$ and $(a,\widetilde{b})$ exceeds $D$: 
  \begin{align*}
    (\widetilde{b}-\widetilde{q_1}) - (p_1-a) 
    &=[(\widetilde{b}-p_1)+(p_1-\widetilde{q_1})] - (p_1-a)\\
    &=\left[(\widetilde{b}-p_1) + \eta_1'+h+\frac{D}{2}+\eta_1\right]-(p_1-a) \\
    &\geq\left[(\widetilde{b}-p_1) + \eta_1'+h+\frac{D}{2}+\eta_1\right]-\eta_1',
    \text{ since }a\in(x,p_1) \\
    &=(\widetilde{b}-p_1)+h+\frac{D}{2}+\eta_1\\
    &>\left(\frac{5D}{8}-\Delta\right)+h+\frac{D}{2}+\eta_1,\text{ since }
    \widetilde{b}\in\J_5 \\
    &=D+\frac{D}{8}+h+\eta_1-\Delta\geq D,
    \text{ for any }\Delta\text{ satisfying }\eqnref{delta-2-wide}.
  \end{align*}
  Therefore, $a>p_1$ and the edge $(a,\widetilde{b})$ does not cross
  $(p_1,\widetilde{q_1})$. Hence arguing similar to $\I_1$, we have
  $\mod{a-(\widetilde{b}+\Delta)}\leq\frac{5D}{8}$. 
\end{case}
This completes the proof for wide crossing.
\end{proof}

In order to complete our analysis of various types of correspondences, we show
now that a flip is not required if there is no wide crossing in $\C$.
\begin{theorem}[No Wide Crossing]\label{thm:no-wide} Let $\C$ be a
correspondence between two compact sets $X,Y\subset\R$ with distortion $D$. If
there are double crossings but not wide, then there exists a value $\Delta\in\R$
such that
$$d_H(X,Y+\Delta)\leq\frac{5D}{8}.$$
\end{theorem}
\begin{proof}
We assume that there are double crossings in $\C$, but none of them are wide;
see \figref{cross-1}. 
\begin{figure}[thb]
  \centering
  \includegraphics[scale=0.9]{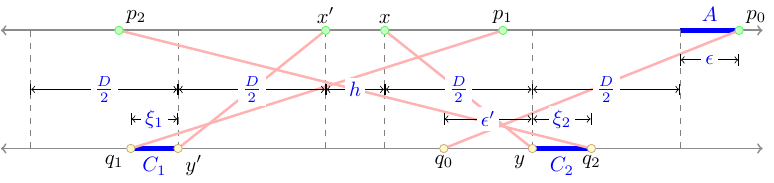}
  \caption[No Wide Crossing]{No wide crossing exists, but there are double
  crossings.}
  \label{fig:cross-1}
\end{figure}
Since there are double crossings, at least one of the following two subsets of
$Y$ is non-empty:
\begin{equation}\label{eqn:C}
  \begin{split}
    C_1&=\left\{b\in Y\cap\left(-\infty,y'\right)\mid
    \exists~a\in X\cap(x,\infty)\mbox{ with }(a,b)\in\C\right\}, \\ 
    C_2&=\left\{b\in Y\cap\left(y,\infty\right)\mid
    \exists~a\in X\cap(-\infty, x')\mbox{ with }(a,b)\in\C\right\}. \\ 
  \end{split}
\end{equation}
We also define $q_1=\min C_1$, $q_2=\max C_2$, and edges
$(p_1,q_1),(p_2,q_2)\in\C$. Let us now denote $\xi_1=y'-q_1$ and $\xi_2=q_2-y$.
If $d_H(X,Y)\leq\frac{5D}{8}$, we take $\Delta=0$. Let us assume the non-trivial
case that $d_H(X,Y)>\frac{5D}{8}$. Since wide crossings are not allowed, we have
from \defref{wide-cross} that $p_2\geq x'-D,p_1\leq p_0$ and from \lemref{cross}
that $h\leq\frac{D}{2}$ . Consequently, the assumption that
$d_H(X,Y)>\frac{5D}{8}$ implies either $\eps>\frac{D}{8}$, $\xi_1>\frac{D}{8}$,
or $\xi_2>\frac{D}{8}$. Without any loss of generality, we also assume that
$\xi_1\geq\xi_2$ so that the translation of $Y$ occurs from left to right.
We now consider the following two cases.
\begin{case}[$\eps\geq\xi_1$]
  This case is equivalent to Case ($1$) of \thmref{no-cross}. In Case ($1$) of
  \thmref{no-cross}, the intervals $\I_1,\I_6$ and $\J_1,\J_5$ use the fact that
  there are no double crossings. We show here that arguments presented for the
  intervals still work in the presence of (non-wide) double crossings. 

  \paragraph{($\pmb{\I_1}$)} For any $a\in\I_1\cap X$ with an edge $(a,b)$, we
  show that $(a,b)$ cannot cross $(x,y)$. And, the rest of the argument for
  $\I_1$ goes through. For any $a\in\I_1\cap X$, we note that $a$ satisfies
  \eqnref{wide-p}, the condition of a wide crossing. When $\eps\leq\frac{D}{4}$,
  we have 
  \begin{align*}
    x'-a&>\frac{D}{2}+\left(\frac{5D}{8}-\Delta\right)
    =D-\Delta+\frac{D}{8}=D-\left(\eps-\frac{D}{8}\right)+\frac{D}{8},
    \text{ from \eqnref{delta-1}} \\
    &=D+\left(\frac{D}{4}-\eps\right)\geq D,
    \text{ since }\eps\leq\frac{D}{4}.
  \end{align*}
  When $\eps>\frac{D}{4}$, we have 
  $$x'-a>\frac{D}{2}+\left(\frac{5D}{8}-\Delta\right)\geq\frac{D}{2}, \text{ as
  }\Delta\leq\frac{5D}{8}.$$  Consequently, an edge $(a,b)\in\C$ cannot be a
  double crossing, otherwise it would become a wide crossing. Hence, $(a,b)$
  cannot cross $(x,y)$.

  \paragraph{($\pmb{\I_6}$)}  For any $\in\I_6\cap X$ with an edge $(a,b)$, we
  argue that $(a,b)$ cannot cross $(x,y)$. And, the rest of the argument for
  $\I_6$ goes through. We have already noted there that $a>\max A$. Hence,
  $(a,b)$ cannot cross $(x,y)$, otherwise it becomes a wide crossing.

  \paragraph{($\pmb{\J_1}$)}  Since $h\leq\frac{D}{2}$ from \lemref{cross}, here
  we can instead consider 
  $$\J_1=\left(-\infty,x'-\frac{5D}{8}-\Delta\right).$$ For any $b\in\J_1\cap Y$
  with an edge $(a,b)$, we argue that $(a,b)$ cannot cross $(x,y)$. Hence, the
  rest of the argument for $\J_1$ presented in Case ($1$) of \thmref{no-cross}
  goes through. We note that $b<q_1$, since
  $\Delta\geq\eps-\frac{D}{8}\geq\xi_1-\frac{D}{8}$. From the definition of the
  set $C_1$, the edge $(a,b)$ cannot be a double crossing.

  \paragraph{($\pmb{\J_5}$)} For any $b\in\J_5\cap Y$ with an edge $(a,b)$, we
  argue that $(a,b)$ cannot cross $(p_0,q_0)$, i.e., $a>p_0$. And, the rest of
  the argument for $\J_5$ goes through. Due to \lemref{standard-1}, we can have
  either $a<x'$, $a\in[x,p_0]$, or $a>p_0$. We arrive at a contradiction when
  assumed the first two.
  
  We first assume that $a<x'$, i.e., $(a,b)$ is a double crossing. This yields a
  contradiction, as $(a,b)$ becomes a wide crossing. When $\eps\leq\frac{D}{4}$,
  we have 
  \begin{align*}
    x'-a&\geq b-y,\text{ from \lemref{cross}} \\
    &>\frac{D}{2}+\eps+\left(\frac{5D}{8}-\Delta\right) 
    \geq\frac{D}{2}+\eps+\frac{5D}{8}-\left(\eps-\frac{D}{8}\right)\geq D.
  \end{align*}
  When $\eps>\frac{D}{4}$, we have 
  \begin{align*}
    x'-a&\geq b-y>\frac{D}{2}+\eps+\left(\frac{5D}{8}-\Delta\right) \\
    &\geq\frac{D}{2}+\eps+\frac{5D}{8}-\left(\eps+\eps'-\frac{3D}{8}\right),
    \text{ from \eqnref{delta-1}} \\
    &=\frac{3D}{2}-\eps'\geq\frac{D}{2},\text{ since }\eps'\leq D
    \text{ from \eqnref{eps}}.
  \end{align*}
  So, $(a,b)$ cannot be a double crossing.

  We now assume that $a\in[x,p_0]$ to arrive at the contradiction that the
  distortion of $(a,b)$ and $(p_0,q_0)$ exceeds $D$:
  \begin{align*}
    (b-q_0)-(p_0-a)&=[(b-p_0)+(p_0-q_0)]-(p_0-a) \\
    &>\left[\left(\frac{5D}{8}-\Delta\right)
    +\left(\eps+\frac{D}{2}+\eps'\right)\right]-(p_0-a) \\
    &\geq\left[\frac{9D}{8}+\eps+\eps'
    -\left(\eps+\eps'-\frac{3D}{8}\right)\right]
    -(p_0-a),\text{ from \eqnref{delta-1}} \\
    &=\frac{12D}{8}-(p_0-a)\geq\frac{12D}{8}-(b-a),\text{ as }b\geq p_0 \\
    &\geq\frac{12D}{8}-\frac{D}{2},\text{ from \lemref{standard-1}} \\&\geq D.
  \end{align*}
  Hence, $a>p_0$, i.e., $(a,b)$ cannot cross $(p_0,q_0)$.
\end{case}

\begin{case}[$\eps\geq\xi_1$]
  This case is equivalent to Case ($2$) of \thmref{no-cross}, and the argument
  is analogous. 
\end{case}
This concludes the proof.
\end{proof}

We conclude this section by showing in \thmref{lb} that the bound of
\thmref{gh} is a tight upper bound in the following sense:
\begin{theorem}[Tightness of the Bound]\label{thm:lb}
For any $0<\eps<\frac{1}{4}$ and $\delta>0$, there exist compact $X,Y\subset\R$
with $d_{GH}(X,Y)=\delta$
and $$d_{H,iso}(X,Y)=\left(\frac{5}{4}-\eps\right)\delta.$$
\end{theorem}
\begin{proof}
It suffices to assume that $\eps=\frac{1}{4(2k+1)}$ for some
$k\in\mathbb{N}$. We now take (sorted)
$$X=\{x',x,x_k,x_{k-1},\cdots,x_1,x_0\}\mbox{ and
}Y=\{y',y_0,y_1,\cdots,y_{k-1},y_k,y\},$$ with distances as shown in
\figref{lb}.  As a result, we also have $(y_i-x) =4i\eps\delta$ and
$(x_k-y_i)=2\delta+4(k-i+1)\eps\delta$, $\forall i\in\{0,1,2,\cdots,k\}$.
\begin{sidewaysfigure}[]
    \centering
    \includegraphics[scale=0.6]{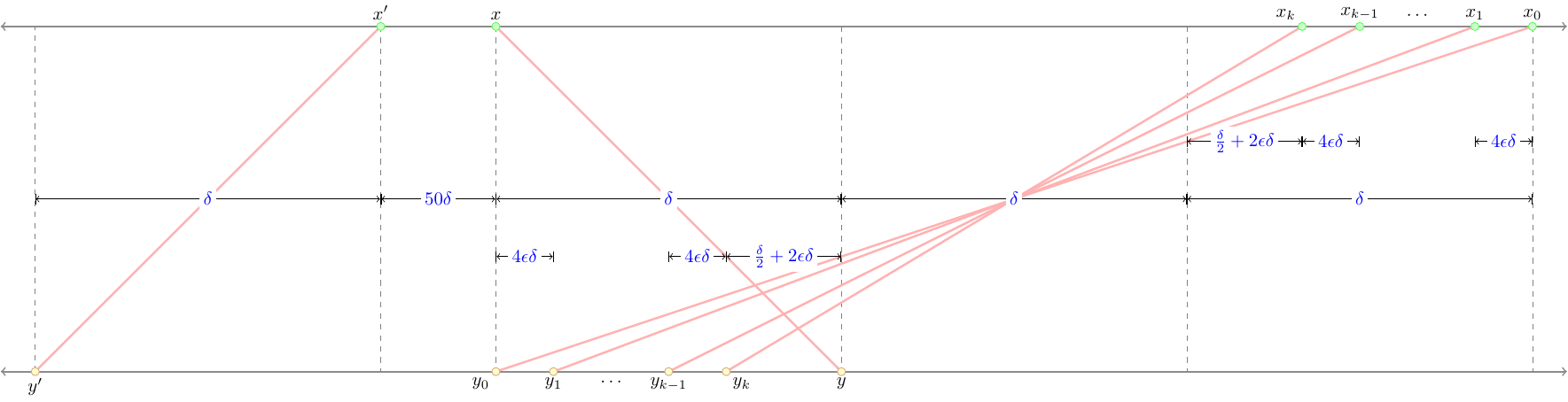}
    \includegraphics[scale=0.6]{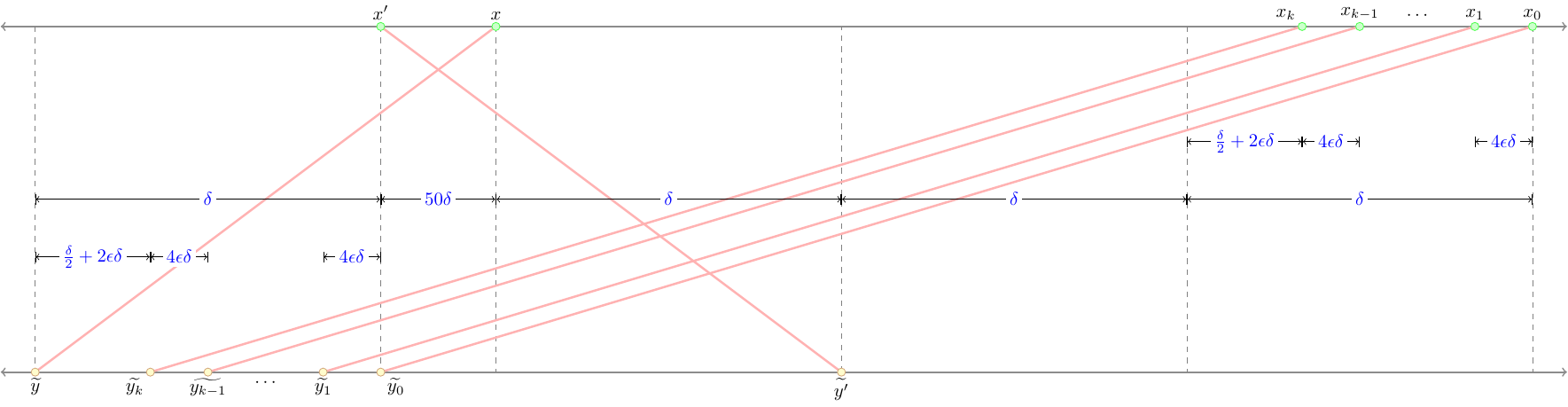}
    \caption[Tight Upper Bound]{The top picture demonstrates the configuration
      of $X$ and $Y$. The correspondence $\C$ is shown using the red edges. In
      the bottom, $X$ and~$\widetilde{Y}$, the reflection of $Y$ about the
      midpoint of $x$ and $x'$, are shown, along with the correspondence $\C$ by
      the red edges.}\label{fig:lb}
\end{sidewaysfigure}
To prove our claim that $d_{H,iso}(X,Y)=\left(\frac{5}{4}-\eps\right)\delta$, we
consider translating both $Y$ and $\widetilde{Y}$, the reflection of $Y$ about
the midpoint of $x$ and $x'$.

When translating $\widetilde{Y}$, we note that the smallest Hausdorff distance
of $\frac{3\delta}{2}$ is achieved for a translation of $\widetilde{Y}$ by an
amount of $\frac{\delta}{2}$ to the right. For this amount of translation,
$\widetilde{y'}$ becomes the midpoint of $x$ and $x_0$, where $\widetilde{y'}$
is the reflection of~$y'$ about the midpoint of $x$ and $x'$. And all the other
points of $\widetilde{Y}$ are at distance at least
$\left(50\delta-\frac{\delta}{2}\right)$ from $x$.

Now, we consider translating $Y$ by an amount $\Delta\in\R$. We first observe
that $d_H(X,Y)=2\delta$, and the distance is attained by $x_0$ and $y$. Now, a
translation of $Y$ to the left is only going to increase the Hausdorff distance
$d_H(X,Y+\Delta)$. Taking this argument one step further we get the following
analysis as we vary~$\Delta$:

If $\Delta\in\left(-\infty,\frac{3\delta}{4}+\eps\delta\right)$, then the pair
$(x_0,y)$ gives
$$d_H(X,Y+\Delta)=2\delta-\Delta>2\delta-\frac{3\delta}{4}-\eps\delta
=\left(\frac{5}{4}-\eps\right)\delta.$$
For $\Delta=\frac{3\delta}{4}+\eps\delta$, we get
$x_0-(y+\Delta)=\left(\frac{5}{4}-\eps\right)\delta$. Also,
$y_k+\Delta-x=\left(\frac{5}{4}-\eps\right)\delta$ and
$x_k-(y_k+\Delta)=\left(\frac{5}{4}-\eps\right)\delta+4\eps\delta$. So,
$d_H(X,Y+\Delta)=\left(\frac{5}{4}-\eps\right)\delta$, which is attained by
$\mod{x_0-y_k}$.

Following this pattern, we conclude that
$d_H(X,Y+\Delta)>\left(\frac{5}{4}-\eps\right)\delta$, except for
$\Delta=\frac{3\delta}{4}+\eps\delta+4i\eps\delta$ for
$i\in\{0,1,2,\cdots,k\}$. Therefore,
$d_{H,iso}(X,Y)=\left(\frac{5}{4}-\eps\right)\delta$. We summarize our analysis
in \tabref{lb}.

{
\setlength{\tabcolsep}{4pt}
\renewcommand{\arraystretch}{1.8}
\begin{table}[tbh]
  \centering
  \footnotesize
  \begin{tabular}{|p{5cm}|p{1.4cm}|p{2.5cm}|p{2cm}|}
    \hline
    $\Delta$ & $\overrightarrow{d}_H(X,Y+\Delta)$ & $\overrightarrow{d}_H(Y+\Delta,X)$ & $d_H(X,Y+\Delta)$\\ 
    \hline
    $\left(-\infty,\frac{3\delta}{4}+\eps\delta\right)$ & $(x_0,y)$ & -- & $>\left(\frac{5}{4}-\eps\right)\delta$ \\ 
    \hline
    $\frac{3\delta}{4}+\eps\delta$ & $(x_0,y)$ & $(y_k,x),(y,x_k)$ & $\left(\frac{5}{4}-\eps\right)\delta$\\
    \hline
    $\left(\frac{3\delta}{4}+\eps\delta,\frac{3\delta}{4}+\eps\delta+4\eps\delta\right)$ & -- & $(y_k,x),(y_k,x_k)$ & $>\left(\frac{5}{4}-\eps\right)\delta$\\
    \hline
    $\frac{3\delta}{4}+\eps\delta+4\eps\delta$ & -- & $(y_{k-1},x),(y_k,x_k)$ & $\left(\frac{5}{4}-\eps\right)\delta$\\
    \hline
    $\cdots$ & $\cdots$ & $\cdots$ & $\cdots$ \\
    \hline
    $\left(\frac{3\delta}{4}+\eps\delta+4i\eps\delta,\frac{3\delta}{4}+\eps\delta+4(i+1)\eps\delta\right)$ & -- & $(y_{k-i},x),(y_{k-i},x_k)$ & $>\left(\frac{5}{4}-\eps\right)\delta$\\
    \hline
    $\frac{3\delta}{4}+\eps\delta+4(i+1)\eps\delta$ & -- & $(y_{k-i-1},x)$, $(y_{k-i},x_k)$ & $\left(\frac{5}{4}-\eps\right)\delta$\\
    \hline
    $\cdots$ & $\cdots$ & $\cdots$ & $\cdots$ \\
    \hline
    $\frac{3\delta}{4}+\eps\delta+4k\eps\delta$ & $(x,y_0)$ & $(y_0,x)$ & $(\frac{5}{4}-\eps)\delta$\\
    \hline
    $\left(\frac{3\delta}{4}+\eps\delta+4k\eps\delta,\infty\right)$ & $(x,y_0)$ & -- & $>\left(\frac{5}{4}-\eps\right)\delta$\\
    \hline     
  \end{tabular}
  \caption[Summary of $d_H(X,Y+\Delta)$]{A summary of
   $d_H(X,Y+\Delta)$ is recorded for $\Delta\in\R$. In the second and third
   columns, the directed Hausdorff distances are achieved for the shown pairs of
   points. The other columns are self-explanatory.}
  \label{tab:lb}
\end{table}
}

With the $d_{H,iso}(X,Y)$ computed, we now define the following correspondence
$\C$ between $X$ and $Y$:
$$\C=\big\{(x_i,y_i)\mid
i\in\{0,1,\cdots,k\}\big\}\cup\big\{(x',y'),(x,y)\big\}.$$ The distortion of
$\C$ is evidently $2\delta$. Moreover, we observe that $\C$ is an optimal
correspondence. Therefore, $d_{GH}(X,Y)=\delta$.
\end{proof}

\section{Conclusions and Future Work}
In our investigation, we focus on approximating Gromov-Hausdorff distance by the
Hausdorff distance for subsets of $\R^1$. The use of $d_{H,iso}$ yields an
approximation algorithm with a tight factor. We do not know, however, if other
algorithms can be devised with a better approximation factor. We believe that
the problem of computing the Gromov-Hausdorff distance or even approximating it
by a factor less than $\frac{5}{4}$ in $\R^1$ is NP-hard. The question of a
polynomial-time approximation algorithm for subsets of $\R^d$ is still open for
$d\geq2$.

\section{Acknowledgments}
The authors would like to thank Yusu Wang and Facundo M\'emoli for hosting
Sushovan Majhi at OSU and for discussions on the project during the visit. We
also thank Helmut Alt for his valuable feedback during his visit at Tulane
University in the Spring of 2019. 

\bibliography{main}
\bibliographystyle{ieeetr}

\newpage
\appendix

\renewcommand{\thesection}{\Alph{section}} 

\section{Additional Lemmas}
\begin{lemma}\label{lem:standard} Let $\C$ be any correspondence between
$X,Y\subset\R^1$ in its standard configuration (see \defref{standard}) with
$Dist(\C)=D$. For $s\in[x',x]$ and $(s,t)\in\C$, we have
$\mod{s-t}\leq\frac{D}{2}$. Consequently, $(s,t)$ does not cross $(x,y)$ or
$(x',y')$.
\end{lemma}
\begin{proof}
  We prove by contradiction. Without loss of generality, let us assume that
  $t>s+\frac{D}{2}$. This implies that
  $$\bigmod{\mod{x'-s}-\mod{y'-t}}=\mod{x'-y'}+\mod{t-s}>\frac{D}{2}+\frac{D}{2}=D$$
  This contradicts the assumption that $Dist(\C)=D$. So,
  $\mod{t-s}\leq\frac{D}{2}$.
\end{proof}
\begin{lemma}\label{lem:standard-1} Let $\C$ be any correspondence between
  $X,Y\subset\R^1$ in its standard configuration (see \defref{standard}) with
  $Dist(\C)=D$. If $(s,t)\in\C$ such that $s>x$ and $t>y$, then
  $\mod{s-t}\leq\frac{D}{2}$.
\end{lemma}
\begin{proof}
    We prove by contradiction. Without loss of generality, let us assume that
    $t>s+\frac{D}{2}$. This implies that
    $$\bigmod{\mod{x'-s}-\mod{y'-t}}=\mod{x'-y'}+\mod{t-s}>\frac{D}{2}+\frac{D}{2}=D$$
    This contradicts the assumption that $Dist(\C)=D$. So,
    $\mod{t-s}\leq\frac{D}{2}$.
\end{proof}

\begin{lemma}\label{lem:AB} Let $\C\in\C(X,Y)$ be a correspondence between
$X,Y\subset\R$ that does not have any double crossings. The sets $A,A'$ and
$B,B'$ are as defined in \eqnref{A} and \eqnref{B}, respectively. Then,
\begin{enumerate}[(a)]
  \item $d_H(X,Y)>\frac{5D}{8}$ implies at least one of the above sets is
  non-empty.
  \item $A,A'$ cannot both be non-empty.
  \item $B,B'$ cannot both be non-empty.
  \item the pairs $A,B'$ and $A',B$ cannot both be non-empty.
\end{enumerate}
\end{lemma}

\begin{proof}
  (a)
  In this case, there must exist either 
  \begin{enumerate}[(i)]
  \item $a_0\in X$ with $\min\limits_{b\in Y}\mod{a_0-b}>\frac{5D}{8}$, or
  \item $b_0\in Y$ with $\min\limits_{a\in X}\mod{a-b_0}>\frac{5D}{8}$.
  \end{enumerate}

  If such an $a_0$ exists, we first observe that $a_0$ cannot belong to $[x',x]$
  due to \lemref{standard} and the definition of $a_0$. From the definition of
  $a_0$, we further note that either $a_0>y+\frac{5D}{8}$ or
  $a_0<y'-\frac{5D}{8}$; see \figref{ub-all} for reference. In either case, for
  an edge $(a_0,t)\in\C$, we now argue that $t\in[y',y]$. To see this, consider
  the first case: $a_0>y+\frac{5D}{8}$. If we assume the contrary, i.e.,
  $y\notin[y',y]$, then we must have $t>y$, since no double crossing is allowed.
  However, \lemref{standard-1} would then imply that
  $\mod{a_0-t}\leq\frac{D}{2}$---a contradiction to the definition of $a_0$. For
  $a_0<y'-\frac{5D}{8}$, we can similarly arrive at a contradiction. Therefore,
  $a_0$ belongs to either $A$ or $A'$. So, either $A$ or $A'$ is non-empty.

  If such a $b_0$ exists, we observe that $b_0\in[y',y]$. We assume the
  contrary, for example, that $b_0>y$ and $(s,b_0)\in\C$ for some $s\in X$.
  \lemref{standard} implies that $s\notin[x',x]$. From the assumption of no
  double crossing, then we must have $s>x$. However, \lemref{standard-1} then
  contradicts the definition of $b_0$. We can arrive at a similar contradiction
  assuming $b_0<y'$. From the definition of $b_0$, we further note that
  $b_0\in\left(x'+\frac{5D}{8},x-\frac{5D}{8}\right)$. Now, we observe that an
  edge $(s,b_0)\in\C$ has to cross either $(x,y)$ or $(x',y')$. Otherwise,
  \lemref{standard} again contradicts the definition of $b_0$. Therefore, $b_0$
  belongs to either $B$ or $B'$. So, either $A$ or $A'$ is non-empty.

  (b) We now note that $A$ and $A'$ cannot be both non-empty; see
  \figref{ub-all}. To see this, we let $a_1\in A$, $a_2\in A'$ with $b_1,b_2\in
  Y\cap[y',y]$ so that $(a_1,b_2),(a_2,b_2)\in\C$. Since $a_1>x+D$ and
  $a_2<x'-D$, we arrive at the following contradiction to the fact that
  $D=Dist(\C)$: 
  $$
  \bigmod{\mod{a_1-a_2}-\mod{b_1-b_2}}=\mod{a_1-a_2}-\mod{b_1-b_2}>D.
  $$

  (c) We also that $B$ and $B'$ cannot be both non-empty. To see this, we let 
  $b_1\in B$, $b_2\in B'$ with $a_1>x$ and $a_2<x'$ so that
  $(a_1,b_2),(a_2,b_2)\in\C$. Since $b_1,b_2\in(y'+D,y-D)$, we arrive at the
  following contradiction: 
  $$
  \bigmod{\mod{a_1-a_2}-\mod{b_1-b_2}}=\mod{a_1-a_2}-\mod{b_1-b_2}>D.
  $$
  
  (d) By the similar argument, each of the pairs $A,B'$ and $A',B$ cannot be
  both non-empty.
\end{proof}

\begin{lemma}\label{lem:eps} Let $\eps,\eps'$, and $D$ be non-negative numbers
satisfying \eqnref{eps}. Then, for any $\Delta$ following \eqnref{delta-1}, we
have $\eps-\eps'+\frac{9D}{8}-\Delta>0$.
\end{lemma}
\begin{proof}
The proof considers two cases: $\eps\leq\frac{D}{4}$ and $\eps>\frac{D}{4}$.

If $\eps\leq\frac{D}{4}$, then $\Delta=\eps-\frac{D}{8}$. Since $\eps'\leq D$,
we get 
$$
\eps-\eps'+\frac{9D}{8}-\Delta=\eps-\eps'+\frac{9D}{8}-\left(\eps-\frac{D}{8}\right)
=\frac{10D}{8}-\eps'>D-\eps'\geq0.
$$

If $\eps>\frac{D}{4}$, then $\Delta\leq\eps+\eps'-\frac{3D}{8}$. In this case,
we have
\begin{align*}
  &\eps-\eps'+\frac{9D}{8}-\Delta \\
  &\geq\eps-\eps'+\frac{9D}{8}-\left(\eps+\eps'-\frac{3D}{8}\right) \\
  &=\frac{12D}{8}-2\eps' \\
  &=\frac{12D}{8}-2\eps'-2\eps+2\eps \\
  &=\frac{12D}{8}-2(\eps'+\eps)+2\eps \\
  &>\frac{12D}{8}-2(\eps'+\eps)+2\times\frac{D}{4},
  \text{ since in this case }\eps>\frac{D}{4} \\
  &=\frac{16D}{8}-2(\eps'+\eps) \\
  &\geq0,\text{ since }\eps+\eps'\leq D.
\end{align*}
In either case, we conclude that $\eps-\eps'+\frac{9D}{8}-\Delta>0$.
\end{proof}

\begin{lemma}\label{lem:eps-1} Let $\eps,\eps'$, and $D$ be non-negative numbers
  satisfying \eqnref{eps}. Then, for any $\Delta$ following \eqnref{delta-1}, we
  have $\frac{D}{4}+\eps+\eps'-2\Delta\geq0$.
  \end{lemma}
  \begin{proof}
  The proof considers two cases: $\eps\leq\frac{D}{4}$ and $\eps>\frac{D}{4}$.
  
  If $\eps\leq\frac{D}{4}$, then $\Delta=\eps-\frac{D}{8}$. Since
  $\eps\leq\frac{D}{2}$, we get 
  $$
  \frac{D}{4}+\eps+\eps'-2\Delta=\frac{D}{4}+\eps+\eps'
  -2\left(\eps-\frac{D}{8}\right)
  =\frac{D}{2}+\eps'-\eps\geq \frac{D}{2}-\eps\geq0.
  $$
  
  If $\eps>\frac{D}{4}$, then $\Delta\leq\eps+\eps'-\frac{3D}{8}$. In this case,
  we have
  \begin{align*}
    \frac{D}{4}+\eps+\eps'-2\Delta&\geq\frac{D}{4}+\eps+\eps'-2\left(\eps+\eps'-\frac{3D}{8}\right) \\
    &=D-(\eps+\eps') \\
    &\geq0,\text{ since }\eps+\eps'\leq D.
  \end{align*}
  In either case, we conclude that $\frac{D}{4}+\eps+\eps'-2\Delta\geq0$.
  \end{proof}  
  
  \begin{lemma}\label{lem:LR} Let $\Delta\geq0$ be a number, and $\C$ a
  correspondence between $X,Y\subset\R$ with distortion $D$. Let
  $$E(\Delta)=\left\{b\in Y\mid
  X\cap\left[b-\frac{5D}{8}+\Delta,b+\frac{5D}{8}+\Delta\right]=\emptyset\right\},
  $$
  $$
  Y_\leq=\left\{b\in Y\mid\text{ there exists an edge }(a,b)\in\C\text{ with
  }a\leq b\right\},$$
  and
  $$Y_\geq=\left\{b\in Y\mid\text{ there exists an edge }(a,b)\in\C\text{ with
  }a\geq b\right\}.
  $$
  Then $E(\Delta)\cap Y_\leq$ and $E(\Delta)\cap Y_\geq$ cannot be both 
  non-empty.
  \end{lemma}
  \begin{proof}
    We prove by contradiction. Let $b_1\in E(\Delta)\cap Y_\leq$ and 
    $b_2\in E(\Delta)\cap Y_\geq$. And, $a_1,a_2\in X$ are such that 
    $(a_1,b_1),(a_2,b_2)\in\C$ with $a_1\leq b_1$ and $a_2\geq b_2$. 
    The distortion of the edges is
    \begin{align*}
        &\mod{\mod{a_1-a_2}-\mod{b_1-b_2}} \\
        \geq&(\max\{a_1,a_2\}-\max\{b_1,b_2\})+(\min\{b_1,b_2\}-\min\{a_1,a_2\})\\
        >&2\left(\frac{5D}{8}+\Delta\right),
        \text{ since }b_1,b_2\in E(\Delta) \\
        >&D.
    \end{align*}
  This is a contradiction.
  \end{proof}

  \begin{lemma}\label{lem:flip-1} Let $\C$ be a correspondence with distortion
    $D$ between $X,Y\subset\R$. Let $\frac{D}{4}\leq h\leq\frac{D}{2}$, where
    $h=x-x'$, and $\Delta$ a number such that  
    \begin{equation}\label{eqn:Delta}
      \begin{cases}
        \Delta\in\left[0,\frac{D}{8}\right],\text{ when }\frac{D}{4}\leq h\leq\frac{3D}{8} \\
        \Delta=h-\frac{3D}{8},\text{ when }\frac{3D}{8}<h\leq\frac{D}{2}.
      \end{cases}
    \end{equation}
    Then, for any 
    $$a\in\left[x'+\frac{D}{8}+\Delta,x-\frac{h}{2}+\frac{D}{16}+\frac{\Delta}{2}\right],$$
    with an edge $(a,b)$, we have
    $\mod{a-(\widetilde{b}+\Delta)}\leq\frac{5D}{8}$.
  \end{lemma}
  \begin{proof}
    Without any loss of generality, let us assume that the mid-point of $x'$ and
    $x$ is the origin, i.e., $\widetilde{b}=-b$. Since
    $0\leq\Delta\leq\frac{D}{8}$ and $h=x-x'>\frac{D}{4}$, we have 
    $$x'\leq x'+\frac{D}{8}+\Delta\leq
    x-\frac{h}{2}+\frac{D}{16}+\frac{\Delta}{2}\leq x.$$ As a result,
    $a\in[x',x]$. We then note from \lemref{standard} that $b\in[y',y]$ and
    $\mod{a-b}\leq\frac{D}{2}$. We consider the following two cases:
  \begin{case}[$a\leq0$]
      Since $\Delta\geq0$, the distance $\mod{a-(\widetilde{b}+\Delta)}$ is
      maximum when $b=a-\frac{D}{2}$. So, the maximum value is
      \begin{align*}
      \mod{(\widetilde{b}+\Delta)-a} &=-b+\Delta-a,
      \text{ as we flip about the origin, }\widetilde{b}=-b \\
      &=-2a+\frac{D}{2}+\Delta \\
      &\leq-2\left(x'+\frac{D}{8}+\Delta\right)+\frac{D}{2}+\Delta \\
      &=2(-x')+\frac{D}{4}-\Delta=2x+\frac{D}{4}-\Delta \\
      &=2\times\frac{h}{2}+\frac{D}{4}-\Delta,
      \text{ since the midpoint of }x'\text{ and }x\text{ is assumed to be }0\\
      &=h+\frac{D}{4}-\Delta\leq\frac{5D}{8},
      \text{ considering both cases of \eqnref{Delta}}.
      \end{align*}
  \end{case}
  \begin{case}[$a>0$]
      The distance $\mod{a-(\widetilde{b}+\Delta)}$ is maximum when
      $b=a+\frac{D}{2}$. The maximum value is, 
      \begin{align*}
      \mod{a-(\widetilde{b}+\Delta)} &=2a+\frac{D}{2}-\Delta \\
      &\leq2\left(x-\frac{h}{2}+\frac{D}{16}+\frac{\Delta}{2}\right)+\frac{D}{2}-\Delta \\
      &=\frac{5D}{8},\text{ since }x=\frac{h}{2}.
      \end{align*}
  \end{case}
  Therefore, in either case, we must have
  $\mod{a-(\widetilde{b}+\Delta)}\leq\frac{5D}{8}$.
\end{proof}

\begin{lemma}\label{lem:wide-no-cross} Let $\C$ be a correspondence between
  $X,Y\subset\R$ with distortion $D$ and a wide crossing edge $(p,q)$. Let
  $(a,b)\in\C$ be an edge such that 
  \begin{equation}\label{eqn:wide-1}
    a<\begin{cases}
    \min A',&\text{ if }A'\neq\emptyset \\
    x'-\frac{D}{2},&\text{ if }\max A>y+\frac{3D}{4}\\
    x'-D,&\text{ otherwise} \end{cases}
  \end{equation}
  or
  \begin{equation}\label{eqn:wide-2}
    a>\begin{cases}
    \max A,&\text{ if }A\neq\emptyset \\
    x+\frac{D}{2},&\text{ if }\min A'<y'-\frac{3D}{4}\\
    x+D,&\text{ otherwise}. \end{cases}\end{equation}
  Then, $(a,b)$ has to be a double crossing.
\end{lemma}
\begin{proof}
We prove the result when $a<x'$ and satisfies \eqnref{wide-1}. For $a>x$
satisfying \eqnref{wide-2}, the argument is exactly the same due to symmetry. We
consider the following three cases depending on the position of $p$. 
\begin{case}[$a\leq p<x'$]
  If we assume the contrary that $(a,b)$ is not a double crossing, then $(a,b)$
  and $(p,q)$ cannot cross. We arrive at the a contradiction in each of the
  following two sub-cases.

  \paragraph{$\pmb{\left(\max A\leq y+\frac{3D}{4}\right)}$} We have from
  \defref{wide-cross} that $p<\min A'$ if $A'\neq\emptyset$ and $p<x'-D$ if
  $A'=\emptyset$. In either case, we have $(x'-p)>D$. In either case, we also
  have from $a\leq p$ and the definition of $A'$ that $b<y'$. Now, the
  distortion of the pair $(a,b)$ and $(p,q)$ is
  \begin{align*}
    (q-b)-(p-a)&=(q-p)+(a-b)\geq(q-p)-\mod{a-b}\\
    &\geq(q-p)-\frac{D}{2},\text{ from \lemref{standard-1}}\\
    &=[(q-y)+(y-x')+(x'-p)]-\frac{D}{2}\geq[(y-x')+(x'-p)]-\frac{D}{2}\\
    &=\left(\frac{D}{2}+h\right)+(x'-p)-\frac{D}{2}
    >\left(\frac{D}{2}+h\right)+D-\frac{D}{2}\geq D.
  \end{align*}
  This is a contradiction. 

  \paragraph{$\pmb{\left(\max A>y+\frac{3D}{4}\right)}$} We have
  $a<x'-\frac{D}{2}$. Let $(p_0,q_0)$ be an edge with $p_0=\max A$ and
  $q_0\in[y',y]$. We denote $\eps=p_0-x-D$ and $\eps'=y-q_0$ so that we have
  \eqnref{eps}. Also, $\max A>y+\frac{3D}{4}$ is equivalent to
  $\eps>\frac{D}{4}$.
  
  If we assume $b\in[y',q_0]$, then the distortion of the pair $(a,b)$ and
  $(p_0,q_0)$ clearly exceeds $D$. If we assume $b\in[q_0,y]$, then the
  distortion of the pair exceeds $D$ again:
  \begin{align*}
    (p_0-a)-(b-q_0)&=[(p_0-x')+(x'-a)]-(b-q_0) \\
    &\geq[(p_0-x')+(x'-a)]-(y-q_0),\text{ since }q_0\leq b\leq y\\
    &=(h+D+\eps)+(x'-a)-\eps'>(h+D+\eps)+\frac{D}{2}-\eps' \\
    &\geq(h+D+\eps)+\frac{D}{2}-(D-\eps),\text{ from \eqnref{eps}} \\
    &=h+2\eps+\frac{D}{2}> D,\text{ as }\eps>\frac{D}{4}.
  \end{align*}
  Now, we assume that $b<y'$. We first note that from the distortion of the pair
  $(a,b)$ and $(p_0,q_0)$ that
  \begin{align*}
    D&\geq(p_0-a)-(q_0-b)=(p_0-x')+(x'-a)-[(q_0-y')+(y'-b)] \\
    &=(D+\eps+h)+(x'-a)-[(D+h-\eps')+(y'-b)] \\
    &=[(x'-a)-(y'-b)]+(\eps+\eps').
  \end{align*}
  So, $(x'-a)-(y'-b)\leq D-(\eps+\eps')$. From the definition of wide crossing,
  we have $p<x'-\frac{D}{2}$. So, the distortion of the pair $(a,b)$ and $(p,q)$
  exceeds $D$:
  \begin{align*}
    (q-b)-(p-a)&=[(q-y')+(y'-b)]-[(x'-a)-(x'-p)] \\
    &=(q-y')+(x'-p)-[(x'-a)-(y'-b)] \\
    &\geq(q-y')+(x'-p)-[D-(\eps+\eps')], \text{ as noted above} \\
    &\geq(y-y')+(x'-p)-[D-(\eps+\eps')] \\
    &>(D+h)+\frac{D}{2}-[D-(\eps+\eps')]=h+\frac{D}{2}+\eps+\eps' \\
    &\geq h+\frac{D}{2}+2\eps,\text{ from \eqnref{eps}} \\
    &>D,\text{ since }\eps>\frac{D}{4}.
  \end{align*}
  This is a contradiction.
\end{case}
\begin{case}[$p<a<x'$]
  If we assume the contrary that $(a,b)$ is not a double crossing, then $(a,b)$
  and $(p,q)$ cross, i.e., $(a,\widetilde{b})$ and $(p,\widetilde{q})$ do not
  cross after flipping $Y$. We arrive at the a contradiction in each of the
  following two sub-cases.
  
  \paragraph{$\pmb{\left(\max A\leq y+\frac{3D}{4}\right)}$} We have $a<\min
  A'$. The definition of $A'$ implies that $b<y'$. We also have $(x'-a)>D$.    
  So, the distortion of the pair $(a,\widetilde{b})$ and $(p,\widetilde{q})$
  exceeds $D$:
  \begin{align*}
    (\widetilde{b}-\widetilde{q})-(a-p)&=(\widetilde{b}-a)+(p-\widetilde{q})
    \geq(\widetilde{b}-a)-\mod{p-\widetilde{q}}\\
    &\geq(\widetilde{b}-a)-\frac{D}{2},\text{ from \lemref{cross}}\\
    &\geq[(y-x')+(x'-a)]-\frac{D}{2}
    =\left(\frac{D}{2}+h\right)+(x'-a)-\frac{D}{2} \\
    &>\left(\frac{D}{2}+h\right)+D-\frac{D}{2}\geq D.
  \end{align*}
  This is a contradiction. 

  \paragraph{$\pmb{\left(\max A>y+\frac{3D}{4}\right)}$} When $a$ and $p$ are
  interchanged, the configuration becomes exactly the same as considered in
  the second sub-case of Case ($1$).
\end{case}
\begin{case}[$p>x$]
  If we assume the contrary that $(a,b)$ is not a double crossing. We arrive at
  the a contradiction in each of the following two sub-cases.

  \paragraph{($\pmb{A\neq\emptyset}$)}  
  From \lemref{AB}, we then have $A'=\emptyset$. As a result, $b<y'$. Since
  $p>x+D$ and $a<x'-\frac{D}{2}$, the distortion of $(a,b)$ and $(p,q)$ exceeds
  $D$.

  \paragraph{($\pmb{A'\neq\emptyset}$)} Since $a<\min A'$, we have from the
  definition of $A'$ that $b<y'$. Since $p>x+\frac{D}{2}$ and $a<x'-D$, the
  distortion of $(a,b)$ and $(p,q)$ exceeds $D$.
\end{case}
\end{proof}
\begin{lemma}\label{lem:wide-no-cross-1} Let $\C$ be a correspondence between
  $X,Y\subset\R$ with distortion $D$ and a wide crossing edge $(p,q)$. Let
  $(a,b)\in\C$ be an edge such that $b<y'-\frac{D}{2}$ (equivalently
  $b>y+\frac{D}{2}$). If $A\neq\emptyset$ (equivalently $A'\neq\emptyset$), then
  $(a,b)$ has to be a double crossing, i.e., $a>x$ (equivalently $a<x'$).
\end{lemma}
\begin{proof}
  We prove the result for $b<y'-\frac{D}{2}$. For $b>y+\frac{D}{2}$, the
  argument is exactly the same due to symmetry. We assume the contrary that
  $(a,b)$ is not a double crossing. Due to \lemref{standard-1}, we then have
  $a<x'$. Depending on the position of $p$, we consider the following sub-cases
  to arrive at a contradiction in each of them.
  \begin{case}[$a\leq p<x'$]
    Since $(a,b)$ does not cross $(x',y')$, we have from \lemref{standard-1}
    that $\mod{a-b}\leq\frac{D}{2}$. So, $b<y'-\frac{D}{2}$ implies that
    $a<x'-\frac{D}{2}$. We can use the argument presented in Case ($1$) of
    \lemref{wide-no-cross} to arrive at the desired contradiction.
  \end{case}
  \begin{case}[$p<a<x'$]
    We note that 
    \begin{align*}
      (\widetilde{b}-a)&=(\widetilde{b}-\widetilde{y})+(\widetilde{y}-x')+(x'-a) \\
      &=(y'-b)+\left(\frac{D}{2}+h\right)+(x'-a) \\
      &=2(y'-b)+\left(\frac{D}{2}+h\right),\text{ since }(x'-a)\geq(y'-b)
      \text{ from \lemref{standard-1}} \\
      &>\frac{3D}{2},\text{ as }b<y'-\frac{D}{2}.
    \end{align*}
    Therefore, we can use Case ($2$) of \lemref{wide-no-cross} to arrive at a
    contradiction.
  \end{case}
  \begin{case}[$p>x$]
    Since we still have $a<x'-\frac{D}{2}$, we can use the sub-case
    ($A\neq\emptyset$) of Case ($3$) from \lemref{wide-no-cross} to arrive at a
    contradiction.
  \end{case}
  So, we arrive at a contradiction in each of the above cases. Therefore,
  $(a,b)$ must be a double crossing.
\end{proof}
\begin{lemma}\label{lem:eta-1} If $(p,q)$ is a wide crossing and
  $\eta_1,\eta_1'$ are as defined in \thmref{wide-cross}, then
  $$
    \eta_1\leq\eta_1'\leq D\text{ and }\eta_1\leq\frac{D}{2}-h.
  $$
\end{lemma}
\begin{proof}
  The reader may use \figref{ub-1} for reference. We first observe that
  $\eta_1\leq~\eta_1'$, otherwise the distortion of $(p_1,q_1)$ and $(x',y')$
  exceeds $D$. We now consider the following two cases depending on the position
  of $p$.
  
  \begin{case}[$p<x'$]
    Similarly, due to the distortion of the
    wide crossing edge $(p,q)$ and $(p_1,q_1)$, we must have
    \begin{align*}
    D &\geq\mod{(p_1-p)-(q-q_1)}=\mod{[(x'-p)+h+\eta_1']-[(q-y)-\eta_1]}\\
    &=(x'-p)+h+\eta_1'-[(q-y)-\eta_1],\text{ as }(x-p')\geq(q-y)+h
    \text{  by \lemref{cross}} \\
    &\geq2h+\eta_1'+\eta_1,\text{ since }(x-p')\geq h+(q-y)
    \text{  by \lemref{cross}}\\
    &\geq2h+2\eta_1,\text{ since }\eta_1'\geq\eta_1\text{ as already noted}.
    \end{align*}
    It implies that $\eta_1\leq\frac{D}{2}-h$. Moreover, we have $\eta_1'\leq D$
    from the third line of the above equation.     
  \end{case}
  \begin{case}[$p>x$]
    If we assumed $p_1>p$, then \lemref{wide-no-cross} would imply $(p_1,q_1)$
    is a double crossing. This is a contradiction. For this reason, we assume
    that $x\leq p_1\leq p$, i.e., the edges $(p,q)$ and $(p_1,q_1)$ cross. As a
    result, the edges $(p,\widetilde{q})$ and $(p_1,\widetilde{q_1})$ do not
    cross after flipping $Y$. From the distortion bound of the pair, we get
    \begin{align*}
      D&\geq(\widetilde{q}-\widetilde{q_1})-(p-p_1)=(p_1-\widetilde{q_1})
      -(p-\widetilde{q}) \\
      &\geq(p_1-\widetilde{q_1})-\mod{p-\widetilde{q}}\geq(p_1-\widetilde{q_1})
      -\left(\frac{D}{2}-h\right),\text{ from \lemref{cross}}\\
      &=[(p_1-x)+(x-\widetilde{y})+(\widetilde{y}-\widetilde{q_1})]-\frac{D}{2}+h \\
      &=\left[\eta_1'+\left(\frac{D}{2}+h\right)+\eta_1\right]-\frac{D}{2}+h \\
      &\geq2\eta_1+2h,\text{ since }\eta_1'\geq\eta_1\text{ as already noted}.
    \end{align*}
    So, $\eta_1\leq\frac{D}{2}-h$. Moreover, we have $\eta_1'\leq D$ from the
    third line of the above equation. 
  \end{case}
 \end{proof}
\end{document}